\documentclass[11pt]{article} %article文档

\usepackage[ruled,linesnumbered]{algorithm2e} %排版算法
\usepackage{amscd} %简单交换图
\usepackage{amsfonts} %字体
\usepackage{amsmath} %数学公式
\usepackage{amssymb} %更多公式符号
\usepackage{amsthm} %定理环境
\usepackage{bm} %允许公式粗体
\usepackage{booktabs} %改善表格
\usepackage{delarray} %矩阵
\usepackage{extarrows} %箭头
\usepackage{geometry} %页面布局
\usepackage{graphicx} %插图
\usepackage{hyperref} %超链接
\usepackage{longtable} %长表格
\usepackage[all]{xy} %交换图
\usepackage[affil-it]{authblk} %标题

\usepackage{adjustbox}
\usepackage{booktabs}
\usepackage{caption}
\usepackage{color}
\usepackage{enumitem} %编号
\usepackage{float}
\usepackage{hyperref}%创建超文本链接和PDF书签
\usepackage{makecell} %表格单元格换行
\usepackage{mathrsfs} %数学字体宏包
\usepackage{mathtools} %数学工具宏包
\usepackage{multirow} %跨行表格宏包
\usepackage{relsize} %缩放宏包

\setlist[enumerate]{itemsep=0pt,parsep=0pt}

\title{On the mod $p$ cohomology for $\mathrm{GL}_2$}
\author{Yitong Wang\thanks{E-mail address: \texttt{yitong.wang@universite-paris-saclay.fr}}}
\date{}
\affil{Universit\'e Paris-Saclay, Laboratoire de Math\'ematiques d'Orsay, 91405, Orsay, France}

\allowdisplaybreaks[4] %多行公式换页

\def\AA{\mathbb{A}}

\def\FF{\mathbb{F}}

\def\NN{\mathbb{N}}

\def\RR{\mathbb{R}}

\def\VV{\mathbb{V}}

\def\ZZ{\mathbb{Z}}

\def\NNN{\mathbb{Z}_{\geq 0}}

\def\cE{\mathcal{E}}

\def\cI{\mathcal{I}}
\def\cJ{\mathcal{J}}

\def\cM{\mathcal{M}}

\def\cO{\mathcal{O}}

\def\fM{\mathfrak{M}}

\def\ff{\mathfrak{f}}
\def\fm{\mathfrak{m}}
\def\fp{\mathfrak{p}}
\def\fq{\mathfrak{q}}
\def\ft{\mathfrak{t}}
\def\fw{\mathfrak{w}}

\def\Ann{\operatorname{Ann}}
\def\Art{\operatorname{Art}}

\def\det{\operatorname{det}}

\def\dim{\operatorname{dim}}

\def\End{\operatorname{End}}
\def\Fil{\operatorname{Fil}}

\def\Gal{\operatorname{Gal}}

\def\GL{\operatorname{GL}}

\def\Hom{\operatorname{Hom}}

\def\JH{\operatorname{JH}}
\def\Ker{\operatorname{Ker}}

\def\M{\operatorname{M}}
\def\Mat{\operatorname{Mat}}

\def\mod{\operatorname{mod}}
\def\Mod{\operatorname{Mod}}

\def\Proj{\operatorname{Proj}}
\def\rad{\operatorname{rad}}
 
\def\reg{\operatorname{reg}}
\def\Rep{\operatorname{Rep}}
\def\Res{\operatorname{Res}}

\def\sgn{\operatorname{sgn}}

\def\soc{\operatorname{soc}}

\def\Spec{\operatorname{Spec}} 
\def\ss{\operatorname{ss}}

\def\Sym{\operatorname{Sym}}

\def\ab{^{\mathrm{ab}}}
\def\Admv{\operatorname{Adm}^{\vee}}
\def\betabar{\overline{\beta}}

\def\corr{\leftrightarrow}

\def\eps{\varepsilon}
\def\eqdef{\overset{\mathrm{def}}{=}}
\def\Fp{\mathbb{F}_p}
\def\Fpbar{\overline{\mathbb{F}_p}}

\def\into{\hookrightarrow}
\def\inv{^{-1}}
\def\ism{\stackrel{\sim}{\rightarrow}}
\def\j{^{(j)}}
\def\ji{^{(j-1)}}
\def\jj{^{(f-1-j)}}
\def\jk{^{(j+1)}}

\def\loc{\textit{loc.cit.}}
\def\mbar{\overline{\mathfrak{M}}}
\def\OK{\mathcal{O}_K}
\def\onto{\twoheadrightarrow}

\def\Qp{\mathbb{Q}_p}

\def\rbar{\overline{r}}
\def\rhobar{\overline{\rho}}
\def\taubar{\overline{\tau}}
\def\tauw{\tau_{\widetilde{w}}}

\def\vpi{\varpi}

\def\we{\widetilde{\underline{W}}} %% extended Weyl group
\def\weee{\widetilde{\underline{W}}} %% in the article it is \wev
\def\x{^{\times}}
\def\xj{^{*(j)}}
\def\Zp{\mathbb{Z}_p}

\newcommand{\bbbra}[1]{\left[#1\right]}
\newcommand{\bbra}[1]{\left(#1\right)}
\newcommand{\bigbra}[1]{\big(#1\big)}

\newcommand{\bra}[1]{(#1)}
\newcommand{\dbra}[1]{(\!(#1)\!)}
\newcommand{\ddbra}[1]{[\![#1]\!]}

\newcommand{\ovl}[1]{\overline{#1}}
\newcommand{\pmat}[1]{\begin{pmatrix}#1\end{pmatrix}}
\newcommand{\set}[1]{\{ #1 \}}
\newcommand{\smat}[1]{\left(\begin{smallmatrix}#1\end{smallmatrix}\right)}
\newcommand{\sset}[1]{\left\{ #1 \right\}}
\newcommand{\un}[1]{\underline{#1}}
\newcommand{\wh}[1]{\widehat{#1}}
\newcommand{\wt}[1]{\widetilde{#1}}

\begin{document}

\theoremstyle{plain} 
\newtheorem{definition}{Definition}[section] %%编号以section为单位
\newtheorem{remark}[definition]{Remark}
\newtheorem{example}[definition]{Example}
\newtheorem{proposition}[definition]{Proposition}
\newtheorem{lemma}[definition]{Lemma}
\newtheorem{corollary}[definition]{Corollary}
\newtheorem{theorem}[definition]{Theorem}
\newtheorem{conjecture}[definition]{Conjecture}

\theoremstyle{plain} 
\newtheorem{ddefinition}{Definition}[subsection] %%编号以subsection为单位
\newtheorem{rremark}[ddefinition]{Remark}
\newtheorem{eexample}[ddefinition]{Example}
\newtheorem{pproposition}[ddefinition]{Proposition}
\newtheorem{llemma}[ddefinition]{Lemma}
\newtheorem{ccorollary}[ddefinition]{Corollary}
\newtheorem{ttheorem}[ddefinition]{Theorem}
\newtheorem{cconjecture}[ddefinition]{Conjecture}

\maketitle
\begin{abstract}
    Let $p$ be a prime number and $F$ a totally real number field unramified at places above $p$. Let $\rbar:{\Gal}(\overline F/F)\rightarrow\GL_2(\ovl{\FF_p})$ be a modular Galois representation which satisfies the Taylor-Wiles hypothesis and some technical genericity assumptions. For $v$ a fixed place of $F$ above $p$, we prove that many of the admissible smooth representations of $\GL_2(F_v)$ over $\ovl{\FF_p}$ associated to $\rbar$ in the corresponding Hecke-eigenspaces of the mod $p$ cohomology have Gelfand--Kirillov dimension $[F_v:\Qp]$. This builds on and extends the work of Breuil-Herzig-Hu-Morra-Schraen in \cite{BHHMS} and Hu-Wang in \cite{HW}, giving a unified proof in all cases ($\rbar$ either semisimple or not at $v$). 

    \hspace{\fill}
    
\end{abstract}

\tableofcontents
\section{Introduction}

\hspace{1.5em}%
Let $p$ be a prime number. The $\mod p$ Langlands correspondence for $\GL_2(\Qp)$ is completely known by the work of \cite{Bre03}, \cite{Col10}, etc. However, the situation becomes much more complicated when we consider $\GL_2(L)$ for $L$ a nontrivial finite extension of $\Qp$. For example, there are many more irreducible admissible smooth representations of $\GL_2(L)$ over $\Fpbar$ and we don't have a classification of these representations (\cite{BP12}).

Motivated by the local-global compatibility results of Emerton (\cite{Eme11}), we study the representations of $\GL_2(L)$ that come from geometry and hope that these representations can realize a $\mod p$ Langlands correspondence for $\GL_2(L)$.

\hspace{\fill}

We introduce the global setup following \cite{BHHMS}. Let $F$ be a totally real number field which is unramified at places above $p$. Let $D$ be a quaternion algebra with center $F$ which is split at places above $p$ and at exactly one infinite place. Let $\FF$ be a sufficiently large finite extension of $\Fp$, which is considered as the coefficient field. For each compact open subgroup $V$ of $(D\otimes_F\AA_F^{\infty})^{\times}$, we denote by $X_V$ the associated smooth projective Shimura curve over $F$. Let $v$ be a fixed place of $F$ above $p$. Let $F_v$ be the completion of $F$ at $v$ and $f\eqdef[F_v:\Qp]$. We consider an admissible smooth representation of $\GL_2(F_v)$ over $\FF$ of the form
\begin{equation}\label{Intro Eq pi}
    \pi\eqdef\lim\limits_{\underset{V_v}{\longrightarrow}}\Hom_{\Gal(\overline{F}/F)}\bigbra{\overline{r},H^1_{\text{\'et}}(X_{V^vV_v}\times_F\overline{F},\FF)},
\end{equation}
where $V^v$ is a fixed compact open subgroup of $(D\otimes_F\AA_F^{\infty,v})^{\times}$, the inductive limit runs over compact open subgroups $V_v$ of $\GL_2(F_v)$, and $\overline{r}:\Gal(\overline{F}/F)\to\GL_2(\FF)$ is a continuous absolutely irreducible representation such that $\pi\neq0$. Let $I_{F_v}$ be the inertia subgroup of $F_v$. Let $k'$ be the quadratic extension of the residue field of $F_v$ and fix an embedding $k'\into\FF$. Let $\omega_{f'}$ be the corresponding Serre's fundamental character of level $f'$ for $f'\in \set{f,2f}$. We make the following assumptions on $\rbar$:
\begin{enumerate}
\item $\rbar|_{G_{F(\!\sqrt[p]{1})}}$ is absolutely irreducible;
\item for $w\!\nmid\! p$ such that either $D$ or $\rbar$ ramifies, the framed deformation ring $R_{\rbar_w}$ of $\rbar_w\eqdef \rbar\vert_{{\rm Gal}(\overline F_w/F_w)}$ over the Witt vectors $W(\FF)$ is formally smooth;
\item for $w\!\mid\!p$, $\rbar\vert_{I_{F_w}}$ is generic in the sense of \cite[\S11]{BP12};
\item $\rbar\vert_{I_{F_v}}$ is isomorphic to one of the following forms up to twist:
\begin{equation}\label{Intro Eq generic}
\begin{aligned}
    &\text{(a)}~\pmat{\omega_f^{\sum_{j=0}^{f-1} (r_{j}+1)p^j}&*\\0&1}~\text{with}~12\leq r_i\leq p-15;\\
    &\text{(b)}~\pmat{\omega_{2f}^{\sum_{j=0}^{f-1} (r_{j}+1)p^j}&0\\0&\omega_{2f}^{p^f(\text{same})}}~\text{with}~13\leq r_0\leq p-14~\text{and}~12\leq r_i\leq p-15~\text{for}~i>0.
\end{aligned}
\end{equation}
\end{enumerate}
We refer to \cite[Introduction]{BHHMS} for the notion of the Gelfand-Kirillov dimension. Our main result is the following:

\begin{theorem}[Corollary \ref{Global Thm main} (iii)]\label{Intro Thm main}
    Keep all the above assumptions on $F,D,\rbar$. Let $V^v=\prod_{w\neq v}V_w$ with $V_w=\GL_2(\cO_{F_w})$ if neither $D$ nor $\rbar$ ramifies at $w$, and $V_w\subseteq1+pM_2(\cO_{F_w})$ is normal in $\GL_2(\cO_{F_w})$ if $w\mid p$ and $w\neq v$. Then for $\pi$ as in (\ref{Intro Eq pi}) we have $\dim_{\GL_2(F_v)}(\pi)=f$.
\end{theorem}

Theorem \ref{Intro Thm main} is proved by \cite[Thm.~1.1]{BHHMS} when $\rbar$ is semisimple at $v$ and is proved by \cite[Thm.~1.1]{HW} in the ``minimal" case (i.e.\,$\pi=\pi_v^{D}(\rbar)$, see \cite[Introduction]{HW}) when $\rbar$ is non-semisimple at $v$. So the new case of Theorem \ref{Intro Thm main} is that we allow arbitrary $\pi$ when $\rbar$ is non-semisimple at $v$. On the one hand, the method of \cite{HW} only works in the non-semisimple case. On the other hand, it turns out that the method of \cite{BHHMS} can be generalized to the non-semisimple case, and this was not noticed before. We adapt the method of \cite{BHHMS} to the non-semisimple case and give a uniform proof of Theorem \ref{Intro Thm main}, see Theorem \ref{Global Thm main}(iii). As an intermediate step, \cite{BHHMS} gives an explicit computation of some potentially crystalline deformation rings using the machinery of Kisin modules developed in \cite{LLHLM18} and \cite{LLHL19} when $\rbar_v$ is semisimple. We generalize the computation of potentially crystalline deformation rings to the non-semisimple case, see Theorem \ref{Comp Thm single} and Theorem \ref{Comp Thm multi-type}.

\subsection*{Organization of the article}

\hspace{1.5em}%
In \S\ref{Sec Pre}, we recall the preliminary notions and results on algebraic groups, tame inertial types and extension graphs. In \S\ref{Sec Kisin}, we recall the machinery of Kisin modules. In \S\ref{Sec Comp}, we use the machinery of Kisin modules to compute explicitly some potentially crystalline deformation rings. In \S\ref{Sec GK}, we recall a result which gives an upper bound for the Gelfand-Kirillov dimensions of some admissible smooth representations of $\GL_2(F_v)$ over $\FF$. In \S\ref{Sec global} we combine all the previous results and prove Theorem \ref{Intro Thm main}.

\subsection*{Notation}

\hspace{1.5em}%%首行缩进
If $F$ is any field, we denote by $G_F\eqdef\Gal(\ovl{F}/F)$ the absolute Galois group of $F$, where $\ovl{F}$ is a fixed separable closure of $F$. If $F$ is a local field, let $I_F\subseteq G_F$ be the inertia subgroup and $W_F\subseteq G_F$ be the Weil group. We normalize Artin's reciprocity map $\Art_F:F\x\ism W_F\ab$ so that uniformizers are sent to geometric Frobenius elements. If $F$ is a number field and $v$ is a finite place of $F$, we write $F_v$ for the completion of $F$ with respect to the place $v$ and write $\cO_{F_v}$ for the ring of integers of $F_v$. We also denote by $\AA_F^{\infty}$ the set of finite ad\`eles of $F$. If $F$ is a perfect field of characteristic $p$, we denote $W(F)$ the ring of Witt vectors of $F$. For $x\in F$, we denote by $[x]\in W(F)$ the Techm\"uller lift of $x$. 

We write $\varepsilon:G_{\Qp}\to\Zp\x$ for the cyclotomic character and $\omega$ its mod $p$ reduction. We normalize Hodge--Tate weights so that $\varepsilon$ has Hodge--Tate weights $1$ at every embedding.

We let $E$ be a finite extension of $\Qp$ with ring of integers $\cO$, uniformizer $\vpi$ and residue field $\FF$. We always assume that $E$ is large enough.

We let $K$ be an unramified extension of $\Qp$ of degree $f$ with ring of integers $\cO_K$ and residue field $k$. We fix an embedding $\sigma_0:k\into\FF$ and let $\sigma_j\eqdef\sigma_0\circ\varphi^j$, where $\varphi:x\mapsto x^p$ is the arithmetic Frobenius on $k$. We identify $\cJ\eqdef\Hom(k,\FF)$ with $\ZZ/f\ZZ$ by $\sigma_j\corr j$.

For $d\in\set{1,2}$, we let $K_d$ be the unramified extension of $K$ of degree $d$ with residue field $k'$ and fix an embedding $\sigma_0':k'\into\FF$ extending $\sigma_0$. Then the Serre's fundamental character of level $df$ is defined as $\omega_{df}:I_K =I_{K_d}\xrightarrow{\Art_{K_d}\inv}\cO_{K'}\x\onto{k'}\x\xhookrightarrow{\sigma_0'}\FF\x.$

If $G$ is a group and $V$ is a representation of $G$ on a finite-dimensional $E$-vector space, we denote by $\ovl{V}$ the reduction modulo $\varpi$ of the semisimplification of a $G$-stable $\cO$-lattice in $V$. If $V$ is a representation of $G$ on a finite-dimensional vector space, we let $\JH(V)$ denote the set of Jordan--H\"older factors of $V$. If $\sigma$ is an irreducible representation of $G$, we let $[V:\sigma]$ be the multiplicity of $\sigma$ in the semisimplification of $V$.\par

For each commutative ring $A$, $a\in A\x$ and $\mu=(\mu_1,\mu_2)\in\ZZ^2$, we write $a^{\mu}$ for the matrix $\smat{a^{\mu_1}&0\\0&a^{\mu_2}}\in\GL_2(A)$. If $s\in S_2$ is a permutation, we let $\dot{s}$ denote the associated permutation matrix and let $\sgn(s)\in\set{\pm1}$ be the signature of $s$.

\section{Preliminaries}\label{Sec Pre}

\hspace{1.5em}%
In this section, we recall some notions related to algebraic groups, the inertial types and the Serre weights that we will use. We follow closely \cite[\S2]{BHHMS} and specialize to the case $n=2$.

\hspace{\fill}

Let $\un{G}$ be the algebraic group $\bbra{ \Res_{\cO_K/\Zp}{\GL_2}_{/\cO_K}}\times_{\Zp}\cO$ with diagonal maximal torus $\un{T}$ and center $\un{Z}$. There is a natural isomorphism $\un{G}\cong\prod_{\cJ}\GL_{2/\cO}$ induced by the ring homomorphism $\cO_K\otimes_{\Zp}\cO\cong\cO^\cJ$ defined by $x\otimes y\mapsto(\sigma_j(x)y)_{j\in\cJ}$.

Let $X^*(\un{T})$ be the character group of $\un{T}$ which is identified with $(\ZZ^2)^f$ via the embeddings $\sigma_j$. For each $\mu\in X^*(\un{T})$, we write $\mu=(\mu_j)_{j\in\cJ}$ with $\mu_j=(\mu_{j,1},\mu_{j,2})\in\ZZ^2$. If $(a_1,a_2)\in\ZZ^2$, we write $(\un{a_1},\un{a_2})$ to denote the element of $X^*(\un{T})$ whose corresponding tuple equals $(a_1,a_2)$ at each place $j\in\cJ$. Let $\eta\eqdef(\un{1},\un{0})$. There is an automorphism $\pi$ on $X^*(\un{T})$ defined by $\pi(\mu)_j=\mu_{j-1}$.

We define the set of \textbf{dominant weights} $X^*_+(\un{T})$, the set of \textbf{$p$-restricted weights} $X^*_1(\un{T})$, $X_{\reg}(\un{T})$, $X^0(\un{T})$, and the \textbf{lowest alcove} $\un{C}_0$ as follows:
\begin{align*}
    X_+^*(\un{T})&\eqdef\sset{\lambda\in X^*(\un{T}):\lambda_{j,1}\geq\lambda_{j,2}~\forall\,j\in\cJ};\\
    X_1(\un{T})&\eqdef\sset{\lambda\in X^*(\un{T}):0\leq\lambda_{j,1}-\lambda_{j,2}\leq p-1~\forall\,j\in\cJ};\\
    X_{\reg}(\un{T})&\eqdef\sset{\lambda\in X^*(\un{T}):0\leq\lambda_{j,1}-\lambda_{j,2}<p-1~\forall\,j\in\cJ};\\
    X^0(\un{T})&\eqdef\sset{\lambda\in X^*(\un{T}):\lambda_{j,1}=\lambda_{j,2}~\forall\,j\in\cJ};\\
    \un{C}_0&\eqdef\sset{\lambda\in X^*(\un{T})\otimes_{\ZZ}\RR\cong(\RR^2)^f:-1<\lambda_{j,1}-\lambda_{j,2}<p-1~\forall\,j\in\cJ}.
\end{align*}
Given $N\in\NNN$ and $\lambda\in\un{C}_0$, we say that $\lambda$ is \textbf{$N$-deep} in $\un{C}_0$ if $N\leq\lambda_{j,1}-\lambda_{j,2}\leq p-2-N$ for all $j\in\cJ$. In particular, the existence of an $N$-deep weight in $\un C_0$ implies $p\geq 2N+2$.

Let $\un{W}$ be the Weyl group of $(\un{G},\un{T})$ which acts on $X^*(\un{T})$. For each $w\in\un{W}$, we write $w_j\in S_2$ to denote its $j$-th component via the identification $\un{W}\cong S_2^f$. Let $\we\cong X^*(\un{T})\rtimes\un{W}$ be the extended Weyl group of $(\un{G},\un{T})$. We denote by $t_{\lambda}$ the image of $\lambda\in X^*(\un{T})$ in $\we$. Hence, an element of $\we$ can be uniquely written as $wt_{\lambda}$ with $w\in\un{W}$ and $\lambda\in X^*(\un{T})$. We identify $\we$ with $(\ZZ^2\rtimes S_2)^f$.

For $\lambda\in X^*(\un{T})$, we have the notion of \textbf{$\lambda$-admissible set} $\Admv(\lambda)$ defined as in \cite[p.14]{BHHMS}. In particular, we have
\begin{equation}\label{Pre Eq Admv}
    \Admv(t_{(\un{2},\un{1})})=\sset{\wt{w}\in\we:\wt{w}_j\in\set{t_{(2,1)},\fw t_{(2,1)},t_{(1,2)}}~\forall\,j\in\cJ},
\end{equation}
where $\fw$ denotes the unique nontrivial element of $S_2$.

\hspace{\fill}

An \textbf{inertial type} of $K$ is a representation $\tau:I_K\to\GL_n(E)$ with open kernel that can be extended to $G_K$. In this work, we let $n=2$. In this case, Henniart associates to $\tau$ a smooth irreducible $\GL_2(\OK)$-representation $\sigma(\tau)$ over $E$ (see the appendix of \cite{BM02}), normalized as in \cite[\S2.1.1]{BM02}. We say that an inertial type is \textbf{tame} if it factors through the tame inertial quotient.

For each pair $(s,\mu)\in\un{W}\times X^*(\un{T})$, we associate to it a tame inertial type $\tau(s,\mu+\eta)$ as in \cite[Def.~2.3.1]{BHHMS} and write $\ovl{\tau}(s,\mu+\eta)$ for its reduction modulo $\varpi$. For a tame inertial type $\tau$, a \textbf{lowest alcove presentation} of $\tau$ is a pair $(s,\mu)\in\un{W}\times\un{C_0}$ such that $\tau\cong\tau(s,\mu+\eta)$, and we say that $\tau$ is \textbf{$N$-generic} for some $N\in\NNN$ if it has a lowest alcove presentation $(s,\mu)$ such that $\mu$ is $N$-deep in $\un{C_0}$. As an example, let $s=(s_0,1,\ldots,1)\in\un{W}$ with $s_0\in S_2$ and $\mu=(\mu_j)_j\in X^*(\un{T})$ with $\mu_j=(r_j+m_j,m_j)\in\ZZ^2$, then we have
\begin{equation}\label{Pre Eq type}
    \tau(s,\mu+\eta)\cong
    \begin{cases}
        \begin{pmatrix}\wt{\omega}_f^{\sum_{j=0}^{f-1} (r_{j}+1)p^j}&0\\0&1\end{pmatrix}\otimes\wt{\omega}_f^{\sum_{j=0}^{f-1}m_jp^j}&\text{ if }s_0=1\\
        \begin{pmatrix}\wt{\omega}_{2f}^{\sum_{j=0}^{f-1} (r_{j}+1)p^j}&0\\0&\wt{\omega}_{2f}^{p^f\sum_{j=0}^{f-1} (r_{j}+1)p^j}\end{pmatrix}\otimes\wt{\omega}_f^{\sum_{j=0}^{f-1}m_jp^j}&\text{ if }s_0=\fw.
    \end{cases}
\end{equation}
Here, for $f'\in\set{f,2f}$, $\wt{\omega}_{f'}:I_K\to\cO\x$ is the Teichm\"{u}ller lift of $\omega_{f'}$. Moreover, $\tau(s,\mu+\eta)$ is $N$-generic if $N\leq r_j\leq p-2-N$ for all $j$.

\hspace{\fill}

A \textbf{Serre weight} of $\GL_2(k)$ is an absolutely irreducible representation of $\GL_2(k)$ over $\FF$. For each $\lambda\in X_1(\un{T})$, we define 
\begin{equation*}
    F(\lambda)\eqdef\bigotimes\limits_{j=0}^{f-1}\bbra{\bbra{\Sym^{\lambda_{j,1}-\lambda_{j,2}}k^2\otimes_k\det^{\lambda_{j,2}}}\otimes_{k,\sigma_j}\FF}.
\end{equation*}
This induces a bijection $F:X_1(\un{T})/(p-\pi)X^0(\un{T})\ism\set{\text{Serre weights of }\GL_2(k)}$. We say that a Serre weight $\sigma$ is \textbf{regular} if $\sigma\cong F(\lambda)$ with $\lambda\in X_{\reg}(\un{T})$, i.e.\,$0\leq\lambda_{j,1}-\lambda_{j,2}<p-1$ for all $j$.

Let $\Lambda_W\eqdef X^*(\un{T})/X^0(\un{T})$. It is identified with $\ZZ^f$ by the isomorphism $[(a_j,b_j)_j]\mapsto(a_j-b_j)_j$. For each $\mu\in X^*(\un{T})$, we define the extension graph 
\begin{equation*}
    \Lambda_W^{\mu}\eqdef\sset{\omega\in\Lambda_W\cong\ZZ^f:0\leq\mu_{j,1}-\mu_{j,2}+\omega_j<p-1~\forall\,j\in\cJ}.
\end{equation*}
As in \cite[p.17]{BHHMS}, there is a map
\begin{equation*}
    \ft_{\mu}:\Lambda_W^{\mu}\to X_{\reg}(\un{T})/(p-\pi)X^0(\un{T}),
\end{equation*}
such that the map $\omega\mapsto F(\ft_{\mu}(\omega))$ gives a bijection between $\Lambda_W^{\mu}$ and the set of regular Serre weights of $\GL_2(k)$ with central character $\mu|_{\un{Z}}$. As an example, suppose that $\mu=(r_j,0)\in X^*(\un{T})$ and $\omega=(2n_j+\delta_j)_j\in\Lambda_W^{\mu}$ with $n_j\in\ZZ,\delta_j\in\set{0,1}$, then we have
\begin{equation*}
    F(\ft_{\mu}(\omega))=\bbra{\bigotimes\limits_{j=0}^{f-1}\bbra{\Sym^{r'_j}k^2\otimes_{k,\sigma_j}\FF}}\otimes_\FF(\sigma_0\circ\det)^{e(\omega)}
\end{equation*}
with
\begin{align*}
    r_j'&=
    \begin{cases}
        r_j+\omega_j&\text{if}~\delta_{j+1}=0\\
        p-2-r_j-\omega_j&\text{if}~\delta_{j+1}\neq0;
    \end{cases}\\
    e(\omega)&=\frac{1}{2}\bbra{\delta_0(p^f-1)+\sum\limits_{j=0}^{f-1}(r_j-r'_j)p^j}.
\end{align*}

\section{Kisin modules}\label{Sec Kisin}

\hspace{1.5em}%
In this section, we review the machinery of Kisin modules that are used to compute the Galois deformation rings. We follow closely \cite{LLHLM18},\cite{LLHL19},\cite{LLHLM20},\cite{LLHLM} as well as \cite[\S3.1]{BHHMS}.

\hspace{\fill}

Throughout this section we fix a $1$-generic tame inertial type $\tau:I_K\to\GL_2(\cO)$ of $K$ and a lowest alcove presentation $(s,\mu)$ for $\tau$.

Let $R$ be a $p$-adically complete Noetherian local $\cO$-algebra and $h\in\NNN$, we define the category of \textbf{Kisin modules} over $R$ of $E(u')$-height $\leq h$ and type $\tau$ as in \cite[Def.~3.1.3]{LLHLM20} with the caveat that we consider modules of rank $2$ as opposed to $3$ in \loc, and denote it by $Y^{[0,h],\tau}(R)$. For each $\fM\in Y^{[0,h],\tau}(R)$, we have the notion of an \textbf{eigenbasis} $\beta$ for $\fM$, defined as in \cite[Def.~3.1.6]{LLHLM20}. For $\fM\in Y^{[0,h],\tau}(R)$ with eigenbasis $\beta$, we define the matrices $A_{\fM,\beta}\j\in\M_2(R\ddbra{v})$ for $0\leq j\leq f-1$ as in \cite[(5.4)]{LLHLM}, where $v\eqdef(u')^{p^f-1}$ if $s_0\cdots s_{f-1}=1$ and $v\eqdef(u')^{p^{2f}-1}$ otherwise. For $\lambda=(\lambda_{j,1},\lambda_{j,2})_j\in X^*_+(\un{T})$ a dominant weight such that $0\leq\lambda_{j,i}\leq h$ for all $0\leq j\leq f-1$ and $i\in\set{1,2}$, we define a subcategory $Y^{\leq\lambda,\tau}(R)$ of $Y^{[0,h],\tau}(R)$ following \cite[\S5]{CL18}. Its objects are Kisin modules $\fM\in Y^{[0,h],\tau}(R)$ such that for some (equivalently, any) eigenbasis $\beta$, we have $A_{\fM,\beta}\j\in\M_2\bbra{(v+p)^{\lambda_{j,2}}R\ddbra{v}}$ and $\det(A_{\fM,\beta}\j)\in R\ddbra{v}\x(v+p)^{\lambda_{j,1}+\lambda_{j,2}}$. In this case, we say that $\fM$ has \textbf{height $\leq\lambda$}. For simplicity, we write $Y^{\leq(a,b)}$ to denote $Y^{\leq(\un{a},\un{b})}$.

We write $\cI(\FF)$ for the Iwahori subgroup of $\GL_2\big(\FF\ddbra{v}\big)$ consisting of the matrices which are upper triangular modulo $v$. For $\ovl{\fM}\in Y^{[0,h],\tau}(\FF)$ and $\wt{w}=(\wt{w}_j)_j\in\we$, we say that $\ovl{\fM}$ has \textbf{shape} $\wt{w}$ if for some (equivalently, any) choice of eigenbasis $\ovl{\beta}$, we have $A^{(j)}_{\ovl{\fM},\ovl{\beta}}\in\cI(\FF)\wt{w}_j\cI(\FF)$ in $\GL_2\big(\FF\dbra{v}\big)$ for each $0\leq j\leq f-1$. Here we regard $\wt{w}_j$ as an element of $\GL_2\big(\FF\dbra{v}\big)$ by the assignment $s_jt_{\mu_j}\mapsto\dot{s}_jv^{\mu_j}$. 

\hspace{\fill}

The property of having a fixed shape is not an open condition, as we will see in Example \ref{Kisin Ex shape} below. Instead, we will use the notion of $\wt{w}$-gauge following \cite{LLHLM}. For simplicity of notation, the following definition is slightly different from \cite[Def.~5.2.6]{LLHLM}.

\begin{definition}\label{Kisin Def w-gauge}
    Let $\fM\in Y^{[0,h],\tau}(R)$ and $\wt{w}=(\wt{w}_j)_j\in\weee$. Write $\wt{w}_j=s_jt_{\nu_j}$ with $s_j\in S_2$ and $\nu_j\in\ZZ^2$. We say that $\fM$ has \textbf{$\wt{w}$-gauge} if it has an eigenbasis $\beta$ such that 
    \begin{enumerate}
        \item $A^{(j)}_{\fM,\beta}(v+p)^{-\nu_j}\dot{s}_j\inv\in\GL_2\bbra{R\bbbra{\frac{1}{v+p}}}$ is lower triangular modulo $\frac{1}{v+p}$;
        \item $A^{(j)}_{\fM,\beta}(v+p)^{-\nu_j}\in\GL_2\bbra{R\bbbra{\frac{1}{v+p}}}$ is upper triangular modulo $\frac{v}{v+p}$
    \end{enumerate}
    for each $0\leq j\leq f-1$. Such a $\beta$ is called a \textbf{$\wt{w}$-gauge basis}. We also say that the matrix $A^{(j)}_{\fM,\beta}$ has \textbf{$\wt{w}_j$-gauge}. 
\end{definition}

\begin{remark}
    Let $\ovl{\fM}\in Y^{[0,h],\tau}(\FF)$ and $\wt{w}\in\un{W}$. If $\ovl{\fM}$ has shape $\wt{w}$, then it has $\wt{w}$-gauge by \cite[Remark~5.2.5]{LLHLM}. In general, $\ovl{\fM}$ has a unique shape, but it could have $\wt{w}$-gauge for many choices of $\wt{w}$.
\end{remark}

\begin{example}\label{Kisin Ex shape}
    Let $\alpha,\beta\in\FF\x$, and $a\in\FF$. 
    We list the gauges and shapes of some matrices in $\GL_n\bigbra{\FF\dbra{v}}$ that will be considered in \S\ref{Sec Comp}.
\begin{figure}[H]
    \centering
    \caption{Gauges and shapes of some matrices}
    \begin{tabular}{|c|c|c|}
        \hline
        Matrix & One choice of gauge & Shape\\
        \hline   
        $\pmat{\alpha v^2&0\\av^2&\beta v}$ & $t_{(2,1)}$-gauge & $t_{(2,1)}$\\
        \hline
        $\pmat{0&\beta v\\\alpha v^2&av}$ & $\fw t_{(2,1)}$-gauge & 
        \makecell[c]{$\fw t_{(2,1)}$ if $a=0$\\~~$t_{(2,1)}$ if $a\neq 0$}\\
        \hline
        $\pmat{\alpha v&0\\0&\beta v^2}$ & $t_{(1,2)}$-gauge & $t_{(1,2)}$\\
        \hline
    \end{tabular}
\end{figure}
\end{example}

The following Proposition comes from \cite[Prop.~5.2.7]{LLHLM} which is a generalization of \cite[Thm.~4.1]{LLHLM18} and \cite[Thm.~4.16]{LLHLM18}.

\begin{proposition}\label{Kisin Prop lift w-gauge}
    Let $h\in\NNN$. Let $\mu$ be $(h+1)$-deep in $\un{C_0}$, $\fM\in Y^{[0,h],\tau}(R)$ and $\wt{w}\in\weee$. Suppose that $\fM/\varpi\fM\in Y^{[0,h],\tau}(R/\varpi R)$ has $\wt{w}$-gauge, then $\fM$ has $\wt{w}$-gauge. Moreover, its $\wt{w}$-gauge basis $\beta$ is unique up to scaling by the group $\un{T}(R)$.
\end{proposition}

\hspace{\fill}

Let $\cO_{\cE,K}$ be the $p$-adic completion of $W(k)\ddbra{v}[1/v]$. It has a Frobenius endomorphism $\varphi$ extending the arithmetic Frobenius on $W(k)$ and such that $\varphi(v)=v^p$. Let $R$ be a complete Noetherian local $\cO$-algebra. The completed tensor product $\cO_{\cE,K}\widehat{\otimes}_{\Zp}R$ is naturally equipped with a Frobenius endomorphism $\varphi$. We write $\Phi\Mod^{\text{\'et}}(R)$ the category of \textbf{\'etale $\varphi$-modules} over $\cO_{\cE,K}\widehat{\otimes}_{\Zp}R$. Its objects are finitely generated projective $\cO_{\cE,K}\widehat{\otimes}_{\Zp}R$-modules $\cM$ together with a $\varphi$-semilinear map $\phi_{\cM}:\cM\to\cM$ such that the image of $\phi_{\cM}$ generates $\cM$ as an $\cO_{\cE,K}\widehat{\otimes}_{\Zp}R$-module.

We fix a compatible system $(p_n)_n$ of $p$-power roots of $(-p)$ in $\ovl{\Qp}$ and define $K_{\infty}\eqdef\bigcup\limits_{n\in\NN}K(p_n)$. Let $\Rep_{G_{K_{\infty}}}(R)$ denote the category of finite free $R$-modules with continuous $R$-linear actions of $G_{K_{\infty}}$. By a result of Fontaine (\cite{Fon90}), generalized by \cite{Dee01} for a version with coefficients, there is an exact anti-equivalence of categories
\begin{equation}\label{Kisin Eq Fontaine}
    \VV_K^*:\Phi\Mod^{\text{\'et}}(R)\ism\Rep_{G_{K_{\infty}}}(R).
\end{equation}

We define a functor $\varepsilon_{\tau}:Y^{[0,h],\tau}(R)\to\Phi\Mod^{\text{\'et}}(R)$ as in \cite[\S5.4]{LLHLM}. Composing it with $\VV_K^*$, we get a functor
\begin{equation}\label{Kisin Eq Tdd}
    T_{dd}^*:Y^{[0,h],\tau}(R)\to\Rep_{G_{K_{\infty}}}(R).
\end{equation}

\section{Galois deformation rings}\label{Sec Comp}

\hspace{1.5em}%
In this section, we compute some potentially crystalline Galois deformation rings explicitly. The semisimple case is already known by \cite[\S4]{BHHMS}. We combine the methods of \cite[\S4]{BHHMS} and \cite[\S3]{Le19} to deal with the non-semisimple case. The main results are Theorem \ref{Comp Thm single} and Theorem \ref{Comp Thm multi-type}.

\hspace{\fill}

Throughout this section we fix a $2$-dimensional Galois representation $\rhobar:G_K\to\GL_2(\FF)$ such that $\rhobar^{\ss}|_{I_K}\cong\taubar(s,\mu)$ (here $\rhobar^{\ss}$ means the semisimplification of $\rhobar$), where
\begin{enumerate}
    \item
    $s_j\neq1$ (hence $s_j=\fw$) if and only if $j=0$ and $\rhobar$ is irreducible;
    \item
    $\mu-\eta$ is $N$-deep in $\un{C}_0$ with $N\geq 12$.
\end{enumerate}
Twisting $\rhobar$ by a power of $\omega_f$ if necessary, we furthermore assume that $\mu_j = (r_j+2,1)\in\ZZ^2$ with $N<r_j+1<p-N$ for all $j$ so that (see (\ref{Pre Eq type}))
\begin{equation}\label{Comp Eq rhobar}
    \rhobar|_{I_K} \cong
    \begin{cases}
        \begin{pmatrix}\omega_f^{\sum_{j=0}^{f-1} (r_{j}+1)p^j}&*\\0&1\end{pmatrix}\otimes \omega&\text{ if $\rhobar$ is reducible}\\
        \begin{pmatrix}\omega_{2f}^{\sum_{j=0}^{f-1} (r_{j}+1)p^j}&0\\0&\omega_{2f}^{p^f\sum_{j=0}^{f-1} (r_{j}+1)p^j}\end{pmatrix}\otimes\omega&\text{ if $\rhobar$ is irreducible.}
    \end{cases}
\end{equation}

Let $\ovl{\cM}$ be the \'etale $\varphi$-module over $k\dbra{v}\otimes_{\Fp}\FF$ such that $\VV_K^*(\ovl{\cM})\cong\rhobar|_{G_{K_{\infty}}}$. By \cite[Prop.~3.1]{Le19}, we have a decomposition $\ovl{\cM}\cong\bigoplus\limits_{j\in\cJ}\ovl{\cM}^{(j)}$ with $\ovl{\cM}^{(j)}=\FF\dbra{v}e_1\j\oplus\FF\dbra{v}e_2\j$ such that the matrices of the Frobenius maps $\phi_{\ovl{\cM}}\j:\ovl{\cM}\j\to\ovl{\cM}\jk$ with respect to the basis $\sset{(e_1\j,e_2\j)_j}$ have the following form:
\begin{enumerate}
\item If $\rhobar$ is reducible, we can take 
\begin{equation}\label{Comp Eq etale phi}
    \Mat(\phi_{\ovl{\cM}}\jj)=\begin{pmatrix}\alpha_jv^{r_j+2}&0\\\alpha_ja_{f-1-j}v^{r_j+2}&\beta_jv\end{pmatrix}
\end{equation}
for some $\alpha_j,\beta_j\in\FF\x$ and $a_j\in\FF$. Note that whether $a_j$ equals $0$ or not is unchanged when we rescale the basis. From now on, we fix a choice of $\alpha_j,\beta_j$ and $a_j$.

\item If $\rhobar$ is irreducible, we can take 
\begin{equation}\label{Comp Eq etale irr}
    \Mat(\phi_{\ovl{\cM}}\jj)=
    \begin{cases}
        \begin{pmatrix}\alpha_jv^{r_j+2}&0\\0&\beta_jv\end{pmatrix}&~\text{if}~j\neq 0\\
        \begin{pmatrix}0&-\beta_jv\\\alpha_jv^{r_j+2}&0\end{pmatrix}&~\text{if}~j=0\\
    \end{cases}
\end{equation}
for some $\alpha_j,\beta_j\in\FF\x$. We fix a choice of $\alpha_j,\beta_j$, and define  $a_j\eqdef0$ for all $j$.
\end{enumerate}
In particular, we see that $\rhobar$ is semisimple if and only if $(a_0,\ldots,a_{f-1})=(0,\ldots,0)$. We denote by $W(\rhobar)$ the set of \textbf{Serre weights of $\rhobar$} defined in \cite[\S3]{BDJ10}. In both cases, by \cite[Prop.~3.2]{Le19} it can be described as
\begin{equation}\label{Comp Eq SW}
    W(\rhobar)=\sset{F(\ft_{\mu-\eta}(b_0,\ldots,b_{f-1})):b_j\in\set{0,\sgn(s_j)}~\text{if}~a_{f-1-j}=0~\text{and}~b_j=0~\text{if}~a_{f-1-j}\neq 0
    }.
\end{equation}

\hspace{\fill}

For each $\wt{w}\in\we$, we define $\wt{w}^*\in\we$ by the formula $((s't_{\mu'})^*)_j\eqdef t_{\mu'_{f-1-j}}s_{f-1-j}^{\prime-1}$ for $s'\in\un{W}$ and $\mu'\in X^*(\un{T})$. Now for each $\wt{w}\in\Admv(t_{(\un{2},\un{1})})=\sset{t_{(2,1)},\fw t_{(2,1)},t_{(1,2)}}^f$ (see (\ref{Pre Eq Admv})), we write $\wt{w}^*=t_\nu w$ for some unique $(w,\nu)\in \un{W}\times X^*(\un{T})$. Then we define the tame inertial type 
\begin{equation*}
    \tauw\eqdef\tau(sw^{-1},\mu-sw^{-1}(\nu))
\end{equation*}
with lowest alcove presentation $(s(\tau),\mu(\tau))\eqdef(sw^{-1},\mu-sw^{-1}(\nu)-\eta)$. In particular, $\tauw$ is $(N-1)$-generic. Explicitly, $s(\tau)_j=w_{j}^{-1}$ except when $j=0$ and $\rhobar$ is irreducible, in which case we have $s(\tau)_0=\fw w_{0}^{-1}$. Also we have
\begin{equation}\label{Comp Eq mu-tau}
    \mu(\tau)_j+\eta_j=
    \begin{cases}
        (r_j,0)&\text{if}~(t_{\nu_j}w_j,s_j)\in\set{(t_{(2,1)},1),\, (t_{(2,1)}\fw,\fw), (t_{(1,2)},\fw)}\\
        (r_j+1,-1)&\text{if}~(t_{\nu_j}w_j,s_j)\in\set{(t_{(2,1)},\fw),\, (t_{(2,1)}\fw ,1), (t_{(1,2)},1)}.
    \end{cases}
\end{equation}

We let $\theta:W(\rhobar)\hookrightarrow\sset{t_{(2,1)},\, t_{(1,2)}}^f$ be the map defined as in \cite[Lemma~4.1.2]{BHHMS} (since a non-semisimple $\rhobar$ has less Serre weights than a semisimple one, $\theta$ is an injection rather than a bijection). By the proof of \loc, it can be defined explicitly as follows: If $\sigma=F(\ft_{\mu-\eta}(b_0,\ldots,b_{f-1}))$ as in (\ref{Comp Eq SW}), then $\theta(\sigma)_{f-1-j}=t_{(1,2)}$ if $b_j=0$ and $\theta(\sigma)_{f-1-j}=t_{(2,1)}$ if $b_j\neq0$. For each $\sigma\in W(\rhobar)$, we define
\begin{equation*}
    X(\sigma)\eqdef\set{\wt{w}\in\Admv(t_{(\un 2,\un{1})}):\wt{w}_j\neq\theta(\sigma)_j~\forall\,j}.
\end{equation*}
We also define
\begin{equation*}
    X(\rhobar)\eqdef\bigcup\limits_{\sigma\in W(\rhobar)}X(\sigma)=\sset{\wt{w}\in\Admv(t_{(\un{2},\un{1})}):\wt{w}_{f-1-j}\neq t_{(1,2)}~\text{if}~a_{f-1-j}\neq 0}.
\end{equation*}
By (\ref{Comp Eq SW}), \cite[Lemma~4.1.2]{BHHMS} and the definition of $\theta$, one can easily check that
\begin{equation*}
    X(\rhobar)=\sset{\wt{w}\in\Admv(t_{(\un{2},\un{1})}):\JH\bbra{\ovl{(\sigma(\tau_{\wt{w}})}\otimes_{\FF}(N_{k|\Fp}\circ\det)}\cap W(\rhobar)\neq\emptyset}.
\end{equation*}

\begin{lemma}\label{Comp Lem exist Kisin}
    Let $\wt{w}\in X(\rhobar)$. Up to isomorphism there exists a unique Kisin module $\mbar\in Y^{\leq(2,1),\tauw}(\FF)\subseteq Y^{\leq(3,0),\tauw}(\FF)$ such that $T_{dd}^*(\mbar)\cong\rhobar|_{G_{K_{\infty}}}$ (see (\ref{Kisin Eq Tdd}) for $T_{dd}^*$).
\end{lemma}

\begin{proof}
    We concentrate on the case that $\rhobar$ is reducible. The irreducible case is similar and can also be treated as in \cite[Lemma~4.1.1]{BHHMS}. Define a Kisin module $\mbar$ over $\FF$ of type $\tauw$ by imposing the matrices of the partial Frobenius maps to be  $\ovl{A}\jj=\Mat(\phi_{\ovl{\cM}}\jj)v^{-(\mu(\tau)_j+\eta_j)}\dot{s}(\tau)_j\in\GL_2(\FF\dbra{v})$, where $\Mat(\phi_{\ovl{\cM}}\jj)\in\M_2(\FF\ddbra{v})$ is the matrix in (\ref{Comp Eq etale phi}), and $s(\tau)$, $\mu(\tau)$ are computed in (\ref{Comp Eq mu-tau}). Explicitly, we have
    \begin{equation}\label{Comp Eq matrix Abar}
    \ovl{A}\jj=
    \begin{cases}
        \pmat{\alpha_jv^2&0\\\alpha_ja_{f-1-j}v^2&\beta_jv}&\text{if}~\wt{w}_{f-1-j}=t_{(2,1)}\\
        \pmat{0&\alpha_jv\\\beta_jv^2&\alpha_ja_{f-1-j}v}&\text{if}~\wt{w}_{f-1-j}=\fw t_{(2,1)}\\
        \pmat{\alpha_jv&0\\0&\beta_jv^2}&\text{if}~\wt{w}_{f-1-j}=t_{(1,2)}.
    \end{cases}
    \end{equation}
    Here we remark that for $\wt{w}\in X(\rhobar)$, we necessarily have $a_{f-1-j}=0$ if $\wt{w}_{f-1-j}=t_{(1,2)}$. By Example \ref{Kisin Ex shape}, the matrix $\ovl{A}\jj$ belongs to $\M_2\bigbra{\FF\ddbra{v}}$ and has shape contained in $\Admv(t_{(2,1)})$ for all $j$, hence $\mbar\in Y^{\leq(2,1),\tau_{\wt{w}}}(\FF)\subseteq Y^{\leq(3,0),\tau_{\wt{w}}}(\FF)$ by Proposition \cite[Prop.~5.4]{CL18}. Moreover, we remark that $\mbar$ has $\wt{w}$-gauge by Example \ref{Kisin Ex shape}.
    
    Now by \cite[Prop.~3.2.1]{LLHLM20}, the matrices of the Frobenius maps of the associated \'etale $\varphi$-module $\varepsilon_{\tauw}(\fM)$ (see \S\ref{Sec Kisin} for $\varepsilon_{\tauw}$) with respect to some basis $\ff$ is given by
    \begin{equation*}  
        \Mat_{\ff}(\phi\jj)=\ovl{A}\jj \dot{s}(\tau)_j^{-1}v^{\mu(\tau)_j+\eta_j}.
    \end{equation*}
    This is the same matrix as in (\ref{Comp Eq etale phi}) by the definition of $A\jj$, hence $T_{dd}^*(\mbar)\cong\rhobar|_{G_{K_{\infty}}}$.
    
    The uniqueness of $\mbar$ follows from \cite[Prop.~3.2.18]{LLHL19}, since $\mbar$ has height in $[0,2]$, $\tauw$ is $(N-1)$-generic and $N-1\geq3$.
\end{proof}

\hspace{\fill}

For each $\wt{w}\in X(\rhobar)$, let $R^{\leq(3,0),\tauw}_{\rhobar}$ denote the maximal reduced, $\cO$-flat quotient of the universal framed deformation ring $R^{\square}_{\rhobar}$ that parametrizes potentially crystalline lifts of $\rhobar$ of Hodge--Tate weights $(3,0)$ or $(2,1)$ in each embedding and tame inertial type $\tauw$ (its existence follows from \cite{Kis08}). For each dominant weight $\lambda\in X_{+}^*(\un{T})$, let $R^{\lambda,\tauw}_{\rhobar}$ denote the maximal reduced, $\cO$-flat quotient of $R^{\square}_{\rhobar}$
that parametrizes potentially crystalline lifts of $\rhobar$ of Hodge--Tate weights $\lambda_j$ in the $j$-th embedding for all $j$ and tame inertial type $\tauw$.

The following result is a generalization of \cite[Prop.~4.2.1]{BHHMS} (where $\rhobar$ was assumed to be semisimple). Here we observe that the Tables 1-3 of \loc\,still work in the non-semisimple case, with the following modification: we have $\ovl{A}\jj=\pmat{\ovl{e_{11}^*}v^2&0\\\ovl{d_{21}}v^2&\ovl{d_{22}^*}v}$ for Table 1 and $\ovl{A}\jj=\pmat{0&\ovl{d_{12}^*}v\\\ovl{d_{21}^*}v^2&\ovl{d_{22}}v}$ for Table 2 (see (\ref{Comp Eq matrix Abar})). Here the element $\ovl{d_{21}}$ of Table 1 (resp.\,$\ovl{d_{22}}$ of Table 2) equals to $0$ if $a_{f-1-j}=0$ and belongs of $\FF\x$ if $a_{f-1-j}\neq0$. We also replace the variable $d_{21}$ in row 3 of Table 1 (resp.\,$d_{22}$ in row 3 of Table 2) with $x_{21}\eqdef d_{21}-[\ovl{d_{21}}]$ (resp.\,$x_{22}\eqdef d_{22}-[\ovl{d_{22}}]$). For ease of presentation, when we quote the tables of \cite{BHHMS}, we omit the reference.

\begin{theorem}\label{Comp Thm single}
    Let $\wt{w}\in X(\rhobar)$. We have an isomorphism
    \begin{equation}\label{Comp Eq Ism single}
        R^{\leq(3,0), \tauw}_{\rhobar}\ddbra{X_1,\dots,X_{2f}}\cong\bbra{R^{\tauw}\Big/\sum\limits_jI^{(j)}}\ddbra{Y_1,\dots,Y_4},
    \end{equation}
    where $R^{\tauw}\eqdef\wh{\bigotimes}_{\cO,0\leq j\leq f-1}R^{(j)}$, and where the rings $R^{(j)}$ and the ideals $I^{(j)}$ of $R^{\tauw}$ are found in Tables 1-3. The irreducible components of $\Spec R^{\leq(3,0), \tauw}_{\rhobar}$ are given by $\Spec R^{\lambda, \tauw}_{\rhobar}$, where $\lambda =(\lambda_j)\in\set{(3,0),(2,1)}^f$.
    
    More precisely, via the isomorphism (\ref{Comp Eq Ism single}), for any choice of $\lambda = (\lambda_j)\in\set{(3,0),(2,1)}^f$ the kernel of the natural surjection $R^{\leq(3,0), \tauw}_{\rhobar}\ddbra{X_1,\dots,X_{2f}}\onto R^{\lambda, \tauw}_{\rhobar}\ddbra{X_1,\dots,X_{2f}}$ is generated by the prime ideal $\sum_{j=0}^{f-1}\fp^{(j),\lambda_{f-1-j}}$ of $R^{\tauw}$, where the ideals $\fp^{(j),\lambda_{f-1-j}}$ of $R^{\tauw}$ are found in Tables 1-3.
    
    Moreover, the special fiber of each $\Spec R^{\lambda, \tauw}_{\rhobar}$ is reduced.
\end{theorem}

\begin{proof}
    We follow the proof of \cite[Prop.~4.2.1]{BHHMS} and use without comment the notation of \loc.
    
    Let $\ovl{\fM}\in Y^{\leq(2,1)}$ such that $T_{dd}^*(\ovl{\fM})\cong\rhobar|_{G_{K_{\infty}}}$ as in Lemma \ref{Comp Lem exist Kisin}. By (\ref{Comp Eq matrix Abar}) and Example \ref{Kisin Ex shape} the eigenbasis $\betabar$ of $\mbar$ is a $\wt{w}$-gauge basis in the sense of Definition \ref{Kisin Def w-gauge}. We modify the definition of $D^{\leq(3,0),\tau}_{\mbar,\betabar}(R)$ appearing in the proof of \cite[Prop.~4.2.1]{BHHMS} by requiring $\beta$ to be a $\wt{w}$-gauge basis instead of a gauge basis. Then by Definition \ref{Kisin Def w-gauge}, for any lift $(\fM,\beta,\jmath)\in D^{\leq(3,0),\tau}_{\mbar,\betabar}(R)$ the corresponding matrices $A\jj$ are given in row $1$ of Tables 1-3, where the entries $c_{11}\j,c_{12}\j,\dots$ are in $R$ satisfying $A\jj\mod\fm_R$ equals $\ovl{A}\jj$. Here we remark that the matrix $A\jj$ in row $1$ of Table 2 has $\fw t_{(2,1)}$-gauge but has shape $t_{(2,1)}$ if $a_{f-1-j}\neq 0$. The determination of the rest the tables are then completely analogous to that  of \cite[Prop.~4.2.1]{BHHMS}.
    
    We have a similar diagram as (5.9) of \cite{LLHLM18}. The vertical map labelled ``f.s." in $\loc$ now corresponds to forgetting the $\wt{w}$-gauge basis on the Kisin modules, and is still formally smooth by Proposition \ref{Kisin Prop lift w-gauge}. Let  $R^{\leq(3,0),\tau,\nabla}_{\ovl{\fM},\ovl{\beta}}$ be the maximal reduced and $\cO$-flat quotient of the ring $R^{\leq(3,0),\tau}_{\ovl{\fM},\ovl{\beta}}/\sum_j(I^{(j),\leq (3,0)}+I^{(j),\nabla})$ as in the proof of \cite[Prop.~4.2.1]{BHHMS}. Then the argument of \cite[Thm.~5.12]{LLHLM18} and \cite[Cor.~5.13]{LLHLM18} goes through and gives an isomorphism
    \begin{equation*}
        R^{\leq(3,0), \tau}_{\rhobar}\ddbra{X_1,\dots,X_{2f}}\cong R^{\leq(3,0),\tau,\nabla}_{\ovl{\fM},\ovl{\beta}}\ddbra{Y_1,\dots,Y_4}.
    \end{equation*}
    
    The rest of the proof and the computations are completely analogous to those of \cite[Prop.~4.2.1]{BHHMS}. Here we notice that the equation (25) in \cite[Prop.~4.2.1]{BHHMS} is guaranteed by the assumption $\wt{w}\in X(\rhobar)$. Finally, the proof of the last statement is completely analogous to that of \cite[Cor.~4.2.6]{BHHMS}.
\end{proof}

\hspace{\fill}

For each $\sigma\in W(\rhobar)$, let $R^{\leq(3,0),\sigma}_{\rhobar}$ denote the maximal reduced, $\cO$-flat quotient of $R^{\square}_{\rhobar}$ that parametrizes potentially crystalline lifts of $\rhobar$ of Hodge--Tate weights $\leq(3,0)$ in each embedding and tame inertial type $\tau$ with $\tau\in X(\sigma)$. This is also the flat closure of $\bigcup\limits_{\wt{w}\in X(\sigma)}\Spec R^{\leq(3,0),\tauw}_{\rhobar}[1/p]$ inside $\Spec R^{\square}_{\rhobar}$. 

We define a bijection $i:\Admv(t_{(\un{2},\un{1})})\to\set{1,2,3}^f$ by letting $i(\wt{w})$ be the $f$-tuple given by
\begin{align*}
    i(\wt{w})_j\eqdef
    \begin{cases}
        1&\text{if}~\wt{w}_j=t_{(2,1)}\\
        2&\text{if}~\wt{w}_j=\fw t_{(2,1)}\\
        3&\text{if}~\wt{w}_j=t_{(1,2)}.
    \end{cases}
\end{align*}
They will be the indices of ideals.

The following result is a generalization of \cite[Prop.~4.3.1]{BHHMS} (where $\rhobar$ was assumed to be semisimple). Here we observe that the Tables 4-5 (of \loc) still work in the non-semisimple case, with the modification that we replace the variables $d_{22}$ in row 3 of Tables 4 and 5 with $x_{22}\eqdef d_{22}-[\ovl{d_{22}}]$.

\begin{theorem}\label{Comp Thm multi-type}
    We have an isomorphism
    \begin{equation}\label{Comp Eq ism multi}
        R^{\leq(3,0),\sigma}_{\rhobar}\ddbra{X_1,\dots,X_{2f}}\cong\bbra{S/\bigcap\limits_{\wt{w}\in X(\sigma)}\sum\limits_{j}I^{(j)}_{\wt{w}}}\ddbra{Y_1,\dots,Y_4},
    \end{equation}
    where $S\eqdef\widehat{\bigotimes}_{\cO,0\leq j\leq f-1}S^{(j)}$, and where the ring $S^{(j)}$ and the ideals $I^{(j)}_{\wt{w}}$ of $S$ are as in Table 4 if $\theta(\sigma)_{f-1-j}=t_{(1,2)}$, whereas the ring $S^{(j)}$ and the ideals $I^{(j)}_{\wt{w}}$ of $S$ are as in Table 5 if $\theta(\sigma)_{f-1-j}=t_{(2,1)}$. The irreducible components of $\Spec R^{\leq(3,0),\sigma}_{\rhobar}$ are given by $\Spec R^{\lambda, \tau_{\wt{w}}}_{\rhobar}$, where $\lambda = (\lambda_j)\in\sset{(3,0), (2,1)}^f$ and $\wt{w}\in X(\sigma)$.
    
    More precisely, via the isomorphism (\ref{Comp Eq ism multi}), for any choice of $\lambda=(\lambda_j)\in\sset{(3,0),(2,1)}^f$ and $\wt{w}\in X(\sigma)$ the kernel of the natural surjection $R^{\leq(3,0),\sigma}_{\rhobar}\ddbra{X_1,\dots,X_{2f}}\onto R^{\lambda,\tau_{\wt{w}}}_{\rhobar}\ddbra{X_1,\dots,X_{2f}}$ is generated by the prime ideal $\sum_{j=0}^{f-1}\fp^{(j),\lambda_{f-1-j}}_{\wt{w}}$ of $S$, where the ideals $\fp^{(j),\lambda_{f-1-j}}_{\wt {w}}$ of $S$ are found in Tables 4-5.
\end{theorem}

\begin{proof}
    We follow the proof of \cite[Prop.~4.3.1]{BHHMS} and use without comment the notation of \loc\  The main difference in the non-semisimple case is that we need to modify the definition of the map $\psi$ of \loc\,and to prove the Claim $1$ of \loc\,in this case.
    
    Recall from (\ref{Comp Eq etale phi}) and (\ref{Comp Eq etale irr}) that there exist $\delta_{12}\j,\delta_{21}\j\in\FF\x$ and $\delta_{22}\j\in\FF$, such that
    \begin{equation*}
        \Mat(\phi_{\ovl{\cM}}\jj)=\begin{pmatrix}\delta_{12}\j v&0\\\delta_{22}\j v&\delta_{21}\j v\end{pmatrix}s_j^{-1}v^{\mu'_j},
    \end{equation*}
    where $\mu_j'\eqdef\mu_j-(1,1)=(r_j+1,0)$. Note that $\delta_{22}\j=0$ if and only if $a_{f-1-j}=0$. From now on, we assume that $\rhobar$ is non-semisimple. The semisimple case is similar and is treated in \cite[Prop.~4.3.1]{BHHMS}. In particular, there exists at least one $j\in\cJ$, such that $\delta_{22}\j\neq 0$. We fix one such $j$ and denote it $j_0$.
    
    Let $\ovl{S}\eqdef S/\bigcap_{\wt{w}\in X(\sigma)}\sum_{j}I^{(j)}_{\wt{w}}$ and consider the \'etale $\varphi$-module $\cM$ over $\cO_{\cE,\ovl{S}}$ given by
    \begin{equation*}
        \Mat(\phi_{\cM}^{(f-1-j)})=
        \begin{pmatrix}
            (v+p)\bigbra{[\delta_{12}^{(j)}]+x_{12}\xj}+c_{12}^{(j)}+\frac{b_{12}^{(j)}}{v}&\frac{1}{v}\big((v+p)d_{11}^{(j)}+c_{11}^{(j)}\big)\\
            (v+p)\bigbra{[\delta_{22}\j]+x_{22}\j}+c_{22}^{(j)}&(v+p)\bigbra{[\delta_{21}^{(j)}]+x_{21}\xj}+c_{21}\j+\frac{b_{21}\j}{v}
        \end{pmatrix}s_{j}^{-1}v^{\mu_j'}
    \end{equation*}
    in a suitable basis, where $b_{21}^{(j)}\eqdef0$ if $\theta(\sigma)_{f-1-j}=t_{(1,2)}$ and  $b_{12}^{(j)}\eqdef0$ if $\theta(\sigma)_{f-1-j}=t_{(2,1)}$. In particular, we see that $\cM\otimes_{\ovl{S}}\FF\cong\ovl{\cM}$.
    
    Fix an $\FF$-basis $\gamma_{\FF}$ of $\VV_K^*(\ovl{\cM})\cong\rhobar|_{G_{K_{\infty}}}$ and fix an $\ovl{S}$-basis $\gamma$ of $\VV_K^*(\cM)$ lifting $\gamma_{\FF}$. Denote $\ovl{S}\ddbra{\un{Y}}\eqdef\ovl{S}\ddbra{Y_1,Y_2,Y_3,Y_4}$. Then the $G_{K_{\infty}}$-representation $\VV_K^*\bigbra{\cM\wh{\otimes}_{\ovl{S}}\ovl{S}\ddbra{\un{Y}}}\cong\VV_K^*(\cM)\wh{\otimes}_{\ovl{S}}\ovl{S}\ddbra{\un{Y}}$ over $\ovl{S}\ddbra{\un{Y}}$ together with the basis $\bbra{1+\smat{Y_1&Y_2\\Y_3&Y_4}}(\gamma\otimes 1)$ give rise to a homomorphism $\psi_0:R^{\square}_{\rhobar|_{G_{K_{\infty}}}}\to\ovl{S}\ddbra{\un{Y}}$, and we extend it to a homomorphism $\psi: R^{\Box}_{\rhobar|_{G_{K_\infty}}}\ddbra{\un{X}',\un{X}''}\rightarrow \ovl{S}\ddbra{\un{Y}}$ by
    \begin{equation*}
        \psi(X'_j)=
        \begin{cases}
            x_{12}\xj&\text{if}~0\leq j<f-1\\
            x_{22}^{(j_0)}&\text{if}~j=f-1;
        \end{cases}\quad
        \psi(X''_j)=
        \begin{cases}
            x_{21}\xj&\text{if}~0\leq j<f-1\\
            Y_4&\text{if}~j=f-1.
        \end{cases}
    \end{equation*}
    
    \paragraph{\textit{Claim.}} The map $\psi: R^{\Box}_{\rhobar|_{G_{K_\infty}}}\ddbra{\un{X}',\un{X}''}\rightarrow \ovl{S}\ddbra{\un{Y}}$ is surjective.
    
    Now we prove the Claim following the proof of Claim 1 in \cite[Prop.~4.3.1]{BHHMS} (which treats the semisimple case).
    
    We check that $\psi$ is injective on reduced tangent vectors, i.e.\ on $\FF[\eps]/(\eps^2)$-points: Let $t_0:\ovl{S}\ddbra{\un{Y}}\rightarrow \FF\rightarrow \FF[\eps]/(\eps^2)$ be the zero vector, where the first map maps all the variables of $\ovl{S}$ and the variables $Y_i$ to $0$. Fix one continuous homomorphism $t:\ovl{S}\ddbra{\un{Y}}\rightarrow \FF[\eps]/(\eps^2)$ such that $t\circ\psi=t_0\circ\psi$. The goal is to prove that $t=t_0$.
    
    Abusing notation, we write $t(b_{ik}^{(j)})=\eps b_{ik}^{(j)}$ for some $b_{ik}^{(j)}\in\FF$ on the right, and similarly $t(c_{ik}^{(j)})=\eps c_{ik}^{(j)}$, $t(d_{ik}^{(j)})=\eps d_{ik}^{(j)}$, $t(x_{ik}^{\ast (j)})=\eps x_{ik}^{(j)}$ if $(i,k)=(2,1)$ or $(1,2)$, $t(x_{22}\j)=\eps x_{22}^{(j)}$, and  $t(Y_{i})=\eps y_{i}$. Since $t\circ\psi=t_0\circ\psi$, evaluating on the variables $\un{X}'$ and $\un{X}''$ we deduce that  
    \begin{equation}\label{eq:c1}
        x_{12}^{(j)}=x_{21}^{(j)}=0~\text{for}~0\leq j<f-1,~y_4=0~\text{and}~x_{22}^{(j_0)}=0.
    \end{equation}
    Moreover, by the definition of $\psi_0$ and using $t\circ\psi=t_0\circ\psi$, there is an isomorphism 
    \begin{equation}\label{eq:c2}
        \lambda:\cM_{\ovl{S}\ddbra{\un{Y}}}\widehat{\otimes}_{\ovl{S}\ddbra{\un{Y}},t}\FF[\eps]/(\eps^2)\ism \cM_{\ovl{S}\ddbra{\un{Y}}}\widehat{\otimes}_{\ovl{S}\ddbra{\un{Y}},t_0}\FF[\eps]/(\eps^2)
    \end{equation}
    of \'etale $\varphi$-modules over $\FF[\eps]/(\eps^2)$ which induces the identity of $\ovl{\cM}$ modulo $\varepsilon$ and
    such that $\VV^*_{K}(\lambda)$ sends the basis $\bbra{1+\eps\smat{y_1&y_2\\ y_3&y_4}}(\gamma\otimes1)$ to $\gamma\otimes 1$ on the corresponding $G_{K_{\infty}}$-representations over $\FF[\eps]/(\eps^2)$. In particular, the isomorphism $\lambda$ is realized by the change of basis (i.e.\ $\varphi$-conjugation) by a matrix of the form $1+\eps M_{f-1-j}\in\GL_2(\cO_{\mathcal{E},\FF[\eps]/(\eps^2)})$ for some $M_{f-1-j}\in\M_2(\cO_{\mathcal{E},\FF})=\M_2\bigbra{\FF\dbra{v}}$. In other words, we have
    \begin{equation}\label{eq:c3}
        \begin{split}
            &(1+\eps M_{j-1})\begin{pmatrix}\delta_{12}^{(j)}&0\\\delta_{22}^{(j)}&\delta_{21}^{(j)}\end{pmatrix}s_j^{-1}v^{\mu'_j}(1-\eps \varphi(M_{j}))\\
            &\qquad=
        \begin{pmatrix}
            \delta_{12}^{(j)}+\eps(x_{12}^{(j)}+c_{12}^{(j)}v^{-1}+b_{12}^{(j)}v^{-2})&\eps(d_{11}^{(j)}v^{-1}+c_{11}^{(j)}v^{-2})\\
            \delta_{22}\j+\eps(x_{22}^{(j)}+c_{22}^{(j)}v^{-1})&\delta_{21}^{(j)}+\eps(x_{21}^{(j)}+c_{21}^{(j)}v^{-1}+b_{21}^{(j)}v^{-2})
        \end{pmatrix}
            s_j^{-1}v^{\mu_j'},
        \end{split}
    \end{equation}
    where we have divided by $v$, and $j$ is considered in $\ZZ/f\ZZ$ as usual.
    
    Exactly as in the proof of \cite[Prop.~4.3.1]{BHHMS}, we can show that $M_j\in\M_2\bigbra{\FF\ddbra{v}}$ for each $j$. Then we write $M_j=\smat{m_{11}\j&m_{12}\j\\m_{21}\j&m_{22}\j}$ with $m_{ik}\j\in\FF\ddbra{v}$. Comparing the coefficients of $\eps$ in (\ref{eq:c3}), multiplying on the right by $v^{-\mu'_j}s_j$ and using that $s=1$ since $\rhobar$ is reducible, we get
    \begin{equation}\label{eq:c5}
    \begin{split}
        &\begin{pmatrix}
            m_{11}\ji\delta_{12}\j+m_{12}\ji\delta_{22}\j&m_{12}\ji\delta_{21}\j\\
            m_{21}\ji\delta_{12}\j+m_{22}\ji\delta_{22}\j&m_{22}\ji\delta_{21}\j
        \end{pmatrix}\\
        -&\begin{pmatrix}
            \delta_{12}\j\varphi\bra{m_{11}\j}&\delta_{12}\j v^{r_j+1}\varphi\bra{m_{12}\j}\\
            \delta_{22}\j\varphi\bra{m_{11}\j}+\delta_{21}\j v^{-(r_j+1)}\varphi\bra{m_{21}\j}&\delta_{22}\j v^{r_j+1}\varphi\bra{m_{12}\j}+\delta_{21}\j\varphi\bra{m_{22}\j}
        \end{pmatrix}\\
        =&\begin{pmatrix}
            x_{12}^{(j)}+c_{12}^{(j)}v^{-1}+b_{12}^{(j)}v^{-2}&d_{11}^{(j)}v^{-1}+c_{11}^{(j)}v^{-2}\\
            x_{22}^{(j)}+c_{22}^{(j)}v^{-1}&x_{21}^{(j)}+c_{21}^{(j)}v^{-1}+b_{21}^{(j)}v^{-2}
        \end{pmatrix}.
    \end{split}
    \end{equation}
    
    Comparing the $(1,1)$,$(1,2)$,$(2,2)$-entries of (\ref{eq:c5}) and using that $M$ is $v$-integral, we deduce that $c_{12}\j=b_{12}\j=d_{11}\j=c_{11}\j=c_{21}\j=b_{21}\j=0$ for all $j$. Comparing the order of $v$ of the $(2,1)$-entries of (\ref{eq:c5}) and using that $2\leq r_j+1\leq p$, we deduce that $v\mid m_{21}\j$ in $\FF\ddbra{v}$ and $c_{22}\j=0$ for all $j$. Comparing the $(1,2)$-entries of (\ref{eq:c5}) and using that $d_{11}\j=c_{11}=0$, we deduce that $m_{12}\j=0$ for all $j$. Specializing the equation (\ref{eq:c5}) at $v=0$ and using that $v\mid m_{21}\j$ in $\FF\ddbra{v}$ and $p-(r_j+1)\geq 1$, we get
    \begin{equation}\label{eq:c6}
    \left\{
    \begin{aligned}
        \delta_{12}\j(m_{11}\ji-m_{11}\j)\big|_{v=0}&=x_{12}\j\\
        \delta_{22}\j(m_{22}\ji-m_{11}\j)\big|_{v=0}&=x_{22}\j\\
        \delta_{21}\j(m_{22}\ji-m_{22}\j)\big|_{v=0}&=x_{21}\j.
    \end{aligned}\right.
    \end{equation}
    Since $x_{12}\j=x_{21}\j=0$ for $0\leq j<f-1$ by (\ref{eq:c1}) and $\delta_{12}\j,\delta_{21}\j\in\FF\x$,  we deduce from (\ref{eq:c6}) that $m_{11}\j\big|_{v=0}$ (resp.\,$m_{22}\j\big|_{v=0}$) does not depend on $j$ and we denote it $m_{11}$ (resp.\,$m_{22}$). By (\ref{eq:c6}) again, we deduce that $x_{12}^{(f-1)}=x_{21}^{(f-1)}=0$. We also have $x_{22}^{(j_0)}=0$ by (\ref{eq:c1}), hence by the second equation of (\ref{eq:c6}) we get $m_{11}=m_{22}$ since $\delta_{22}^{(j_0)}\in\FF\x$ by our choice of $j_0$. Then by the second equation of (\ref{eq:c6}) again, we get $x_{22}\j=0$ for all $j$.
    
    As a result, the right-hand side of (\ref{eq:c5}) vanishes, hence we conclude that $(M_{f-1-j})_j\in \End_{\varphi\text{-mod}}(\ovl{\cM})$, and we denote this endomorphism by $\xi$. Since the isomorphism (\ref{eq:c2}) satisfies
    \begin{equation*}
        \bigbra{1+\eps \VV^*_K(\xi)} \bbbra{\bbra{1+\eps\begin{pmatrix}y_1&y_2\\ y_3&y_4\end{pmatrix}}(\gamma\otimes 1)} = \gamma \otimes 1,
    \end{equation*}
    it follows that $\VV^*_K(\xi)=-\smat{y_1&y_2\\y_3&y_4}$ with respect to the basis $\gamma_\FF$. Moreover, we have $\End_{\varphi\text{-mod}}(\ovl{\cM})\cong\End_{G_{K_\infty}}(\rhobar|_{K_\infty})\cong\End_{G_K}(\rhobar)\cong\FF$, where the second isomorphism follows from \cite[Lemma~3.2.8]{BHHMS} and the last isomorphism follows from the fact that $\rhobar$ is non-semisimple. As $y_4=0$ by (\ref{eq:c1}) we conclude from the formula for $\VV^*_K(\xi)$ that $y_i=0$ for all $i$. Therefore, we conclude that $t=t_0$, completing the proof of the Claim.
    
    \hspace{\fill}
    
    The rest of the proof and the computations are completely analogous to that of \cite[Prop.~4.3.1]{BHHMS} using Theorem \ref{Comp Thm single}.
\end{proof}

We let $\fq_1^{(j),(2,1)},\fq_2^{(j),(2,1)},\fq_3^{(j),(2,1)}$ be the ideals of $S\j$ defined as in \cite[Lemma~4.3.2]{BHHMS}. The following result is a generalization of \cite[Prop.~4.3.3]{BHHMS} (where $\rhobar$ was assumed to be semisimple).

\begin{proposition}\label{Comp Prop p}
    Keep the assumption of Theorem \ref{Comp Thm multi-type}. Fix $0\leq j\leq f-1$ and $\wt{w}\in X(\sigma)$ such that $i(\wt{w})_{f-1-j}=2$. Then we have $p\in\fq_1^{(j),(2,1)}\cap\fq_2^{(j),(2,1)}+\fp_{\wt{w}}^{(j),(3,0)}$ if $\theta(\sigma)_{f-1-j}=t_{(1,2)}$ and $p\in\fq_2^{(j),(2,1)}\cap\fq_3^{(j),(2,1)}+\fp_{\wt{w}}^{(j),(3,0)}$ if $\theta(\sigma)_{f-1-j}=t_{(2,1)}$.
\end{proposition}

\begin{proof}
    The proof is purely computational and is completely analogous to that of \cite[Prop.~4.3.3]{BHHMS}.
\end{proof}

\section{An upper bound of the Gelfand-Kirillov dimension}\label{Sec GK}

\hspace{1.5em}%
In this section, we recall a criterion which gives an upper bound for the Gelfand-Kirillov dimensions of certain admissible smooth representations of $\GL_2(K)$ over $\FF$. This criterion is given in \cite[\S6]{BHHMS} and is used in the semisimple case. We show that it works in the non-semisimple case as well using the results of \cite[\S4]{HW}, see Theorem \ref{GK Thm upper}. 

\hspace{\fill}

Let $I\subseteq\GL_2(\cO_K)$ be the subgroup of upper triangular matrices modulo $p$. Let $I_1\subseteq I$ be the subgroup of upper unipotent matrices modulo $p$. Let $K_1\eqdef1+pM_2(\cO_K)$. Let $Z$ be the center of $\GL_2(K)$ and let $Z_1\eqdef Z\cap K_1$. Let $\fm_{K_1}$ denote the unique maximal ideal of the Iwasawa algebra $\FF\ddbra{K_1/Z_1}$. 
The ideal of $\FF\ddbra{\GL_2(\cO_K)/Z_1}$ generated by $\fm_{K_1}$ under the natural inclusion $\FF\ddbra{K_1/Z_1}\into\FF\ddbra{\GL_2(\cO_K)/Z_1}$ is also denoted by $\fm_{K_1}$ when there is no possible confusion. Denote $\Gamma\eqdef\GL_2(k)$ and $\wt{\Gamma}\eqdef\FF\ddbra{\GL_2(\cO_K)/Z_1}/\fm_{K_1}^2$.

Let $\rhobar:G_K\to\GL_2(\FF)$ be as in  (\ref{Comp Eq rhobar}). Recall that $\cite[\S4.1]{HW}$ constructs a finite dimensional representation $\wt{D}_0(\rhobar)$ of (the non-commutative ring) $\wt{\Gamma}$ over $\FF$ (generalizing the constructions of \cite[\S13]{BP12}) characterized by the following properties:
\begin{enumerate}
        \item $\soc_{\wt{\Gamma}}\wt{D}_0(\rhobar)=\bigoplus\nolimits_{\sigma\in W(\rhobar)}\sigma$;
        \item for each $\sigma\in W(\rhobar)$, we have  $[\wt{D}_0(\rhobar):\sigma]=1$;
        \item $\wt{D}_0(\rhobar)$ is maximal with respect to properties (i) and (ii).
    \end{enumerate}
We have a decomposition of $\wt{\Gamma}$-representations $\wt{D}_0(\rhobar)=\bigoplus\nolimits_{\sigma\in W(\rhobar)}\wt{D}_{0,\sigma}(\rhobar)$, where each $\wt{D}_{0,\sigma}(\rhobar)$ satisfies $\soc_{\wt{\Gamma}}\wt{D}_{0,\sigma}(\rhobar)=\sigma$. We have $\wt{D}_0(\rhobar)^{K_1}\cong D_0(\rhobar)$, which is the representation of $\Gamma$ over $\FF$ defined in \cite[\S13]{BP12}. We also define the $I$-representation $D_1(\rhobar)\eqdef D_0(\rhobar)^{I_1}=\wt{D}_0(\rhobar)^{I_1}$.

By \cite[Thm.~4.6]{HW}, the representation $\wt{D}_0(\rhobar)$ of $\wt{\Gamma}$ over $\FF$ is multiplicity-free. Moreover, one can describe explicitly the structure of each $\wt{D}_{0,\sigma}(\rhobar)$ for $\sigma\in W(\rhobar)$. We use the notation of \S\ref{Sec Pre}. Let $\lambda\in X_1(\un{T})$ be such that $\sigma\cong F(\lambda)$. By (\ref{Comp Eq SW}) and \cite[Lemma~2.4.4]{BHHMS} there exist a subset $J_{\rhobar,\sigma}\subseteq\sset{0,\ldots,f-1}$ and elements $\varepsilon_j\in\sset{\pm 1}$ for each $j\in J_{\rhobar,\sigma}$, such that 
\begin{equation*}
    W(\rhobar)=\sset{F(\ft_{\lambda}(b_0,\ldots,b_{f-1})):b_j\in\set{0,\varepsilon_j}~\text{if}~j\in J_{\rhobar,\sigma}~\text{and}~b_j=0~\text{if}~j\notin J_{\rhobar,\sigma}}.
\end{equation*}
Then by \cite[Lemma~4.8]{HW} and \cite[Cor.~2.35]{HW}, the multiplicity-free representation $\wt{D}_{0,\sigma}(\rhobar)$ has Jordan--H\"older factors
\begin{equation}\label{GK Eq JH}
    \JH(\wt{D}_{0,\sigma}(\rhobar))=\sset{\sigma_{\un{a}}\eqdef F(t_{\lambda}(a_0,\ldots,a_{f-1})):a_j\in\ZZ,~\sgn(a_j)\neq\varepsilon_j~\text{for}~j\in J_{\rhobar,\sigma},~\sum\limits_{j=0}^{f-1}\bbbra{\frac{a_j}{2}}\leq 1}
\end{equation}
and its submodule structure is determined as follows: the unique subrepresentation of $\wt{D}_{0,\sigma}(\rhobar)$ with cosocle $\sigma_{\un{a}}$ has constituents $\sigma_{\un{b}}$ for all $\un{b}$ such that each $b_j$ is between $0$ and $a_j$. Here for $x\in\RR$, we denote by $[x]$ the largest integer which is smaller than or equal to $x$.

Now we give a generalization of \cite[Thm.~6.4.7]{BHHMS} where $\rhobar$ was assumed to be semisimple. This gives an upper bound for the Gelfand-Kirillov dimension $\dim_{\GL_2(K)}(\pi)$ of some admissible smooth representations $\pi$ of $\GL_2(K)$ over $\FF$. We refer to \cite[\S5]{BHHMS} for the notion of the Gelfand-Kirillov dimension.

\begin{theorem}\label{GK Thm upper}
    Let $\pi$ be an admissible smooth representation of $\GL_2(K)$ over $\FF$ with a central character. Assume that:
    \begin{enumerate}
        \item
        we have $\JH(\soc_{\GL_2(\cO_K)}(\pi))=W(\rhobar)$ up to multiplicity;
        \item
        for each $\sigma\in W(\rhobar)$, we have $[\pi[\fm_{K_1}^2]|_{\GL_2(\cO_K)}:\sigma]=[\soc_{\GL_2(\cO_K)}\pi:\sigma]$;
        \item
        we have $\JH(\pi^{I_1})=\JH(D_1(\rhobar))$ up to multiplicity (as $I$-representations). 
    \end{enumerate}
    Then $\dim_{\GL_2(K)}(\pi)\leq f$.
\end{theorem}

\begin{proof}
    The proof is analogous to that of \cite[Thm.~6.4.7]{BHHMS}. For $\rhobar$ not necessarily semisimple, the condition (a) of \cite[\S6.4]{BHHMS} is guaranteed by \cite[Cor.~6.3.13.(i)]{BHHMS}, whose proof works for arbitrary $\rhobar$. In particular, we can take $V=\bigoplus\nolimits_{\sigma\in W(\rhobar)}\wt{D}_{0,\sigma}(\rhobar)^{m_{\sigma}}$ in \loc\ The condition (b) of \cite[\S6.4]{BHHMS} is guaranteed by (\ref{GK Eq JH}) and the fact that $\wt{D}_{0,\sigma}(\rhobar)$ is multiplicity-free. Finally the condition (c) of \cite[\S6.4]{BHHMS} is a consequence of \cite[Lemma~6.4.3]{BHHMS}.
\end{proof}
\section{Global applications}\label{Sec global}

\hspace{1.5em}%
In this section, we use the machinery of patching functors introduced by \cite{EGS15} to prove the main global results: Theorem \ref{Global Thm main}. We follow closely \cite[\S8]{BHHMS} which deals with the semisimple case.

\hspace{\fill}

We assume $p>5$ and $E$ unramified, so that ${\cO}=W(\FF)$. We fix $F$ a totally real number field in which $p$ is unramified. We denote by $S_p$ the set of places of $F$ above $p$. We fix a place $v\in S_p$. Then exactly as in \cite[\S8.1]{BHHMS}, we fix a quaternion algebra $D$ with center $F$, a continuous Galois representation $\rbar:G_F\to\GL_2(\FF)$, a continuous character $\psi:G_F\to W(\FF)\x$, a finite place $w_1$, a finite set of places $S$ of $F$, a compact open subgroup $U=\prod_w U_w\subseteq(D\otimes_{F}\AA_F^{\infty})\x$, a tame inertial type $\tau_w$ of $F_w$ for each $w\in S_p\setminus\set{v}$ and a $\GL_2({\cO}_{F_{w}})$-invariant lattice $\sigma^0(\tau_w^\vee)$ in $\sigma(\tau_w^\vee)=\sigma(\tau_w)^\vee$ for each $w\in S_p\setminus\set{v}$, except that we allow $\rbar_v\eqdef\rbar|_{G_{F_v}}$ to be non-semisimple, see (\ref{Intro Eq generic}).

Then as in \cite[\S6]{EGS15} (see also \cite[p.88]{BHHMS}) we can define a \textbf{patching functor} $M_{\infty}$ from the category of continuous representations $\sigma_v$ of $U_v\cong\GL_2({\cO}_{F_{v}})$ on finite type $W(\FF)$-modules with central character $\psi|_{I_{F_v}}^{-1}\circ\Art_{F_v}|_{\cO_{F_v}\x}$ to the category of finite type $R_\infty$-modules, where $R_\infty\cong R_{\rbar_v}^{\psi_v}\ddbra{X_1,\dots,X_g}$ for some integer $g$. Here $R_{\rbar_v}^{\psi_v}$ denotes the framed deformation ring of $\rbar_v$ with fixed determinant $\varepsilon^{-1}\psi_v$ and $\psi_v\eqdef\psi|_{G_{F_v}}$. We denote by ${\fm}_\infty$ the maximal ideal of $R_\infty$.

We keep the notation of \S\ref{Sec Pre} and \S\ref{Sec GK} with $K\eqdef F_v$, so that $k$ is the residue field of $F_v$. In particular we have $\Gamma=\GL_2(k)$ and $\wt{\Gamma}=\FF\ddbra{\GL_2(\cO_{F_v})/Z_1}/\fm_{K_1}^2.$ For $\sigma$ a Serre weight, we let $P_{\sigma}\eqdef\Proj_{\Gamma}\sigma$ be the projective envelope of $\sigma$ in the category of $\FF[\Gamma]$-modules, $\wt{P}_{\sigma}$ be the projective $\cO[\Gamma]$-module lifting $P_{\sigma}$ and $\Proj_{\wt{\Gamma}}\sigma$ be the projective envelope of $\sigma$ in the category of $\wt{\Gamma}$-modules. If $A$ is a ring and $M$ is an $A$-module, we define the \textbf{scheme-theoretic support} of $M$ to be the quotient $A/\Ann_A(M)$. For each $\tau$ a tame inertial type and $\lambda=(a_j,b_j)_j\in X^*_+(\un{T})$, we define $R_{\infty}^{\lambda,\tau}\eqdef R_{\infty}\otimes_{R_{\rbar_v^{\vee}}}R_{\rbar_v^{\vee}}^{\lambda,\tau}$, where $R_{\rbar_v^{\vee}}^{\lambda,\tau}$ parametrizes potentially crystalline lifts of $\rbar_v^{\vee}$ of Hodge--Tate weights $(a_j,b_j)$ in the $j$-th embedding and inertial type $\tau$. When $a_j=a$ and $b_j=b$ for all $j$ we write $R_{\infty}^{(a,b),\tau}$. The following proposition is a generalization of \cite[Prop.~8.2.3]{BHHMS} and \cite[Prop.~8.2.6]{BHHMS} (where $\rbar_v$ was assumed to be semisimple).

\begin{proposition}\label{Global Prop type}\label{Global Prop Proj}
    There exists an integer $r\geq 1$ such that 
    \begin{enumerate}
        \item 
        for all $\sigma_v\in W(\rbar_v^{\vee})$ the module $M_{\infty}(\sigma_v)$ is free of rank $r$ over its scheme-theoretic support, which is formally smooth over $\FF$.
        \item 
        for all tame inertial type $\tau$ such that $\JH(\ovl{\sigma(\tau)})\cap W(\rbar_v^{\vee})\neq\emptyset$ and all $\GL_2(\cO_{F_v})$-invariant $W(\FF)$-lattices $\sigma^0(\tau)$ in $\sigma(\tau)$ with irreducible cosocle, the module $M_{\infty}(\sigma^0(\tau))$ is free of rank $r$ over its scheme-theoretic support $R_{\infty}^{(1,0),\tau}$, which is a domain.
        \item
        for all $\sigma_v\in W(\rbar_v^{\vee})$, the module $M_{\infty}(\wt{P}_{\sigma_v})$ is free of rank $r$ over $R_{\infty}/\cap_{\tau}\fp_{\tau}$, where $\tau$ runs over all tame inertial types such that $\sigma_v\in\JH(\ovl{\sigma(\tau)})$ and $\fp_{\tau}$ is the prime ideal $\Ker(R_{\infty}\onto R_{\infty}^{(1,0),\tau})$ of $R_{\infty}$.
    \end{enumerate}
\end{proposition}

\begin{proof}
    The proof of (i) and (ii) is analogous to the one of \cite[Prop.~8.2.3]{BHHMS}. In the case $|\cJ|=2$ (see the fourth paragraph of the proof of $\loc$), we used the ``connectedness" of $W(\rbar_v^{\vee})$ by non-split extensions to deduce that $M_{\infty}(\sigma)$ has the same rank over its scheme-theoretic support for $\sigma\in W(\rbar_v^{\vee})$. In general the Serre weights in $W(\rbar_v^{\vee})$ can still be ``connected" by non-split extensions by (\ref{Comp Eq SW}) and \cite[Lemma~2.4.6]{BHHMS}.

    Now we prove (iii). The strategy of the proof is very close to the one of \cite[Thm.~4.9]{Le19} which treats the case $r=1$. We freely use the notation from \loc
    
    First we show that $M_{\infty}(R_{\mu}/\Fil_{\otimes}^2R_{\mu})$ is free of rank $r$ over its scheme-theoretic support (see \cite[Lemma~4.3]{Le19}), where $R_{\mu}$ is the same as $P_{\sigma_v}$ and $\Fil_{\otimes}^2R_{\mu}$ is a certain submodule of $R_{\mu}$ defined in \cite[\S3]{LMS20}. The argument of \cite[Lemma~4.3]{Le19} gives a tame inertial type $\tau$ and a $\GL_2(\cO_{F_v})$-invariant $W(\FF)$-lattice $\sigma^0(\tau)$ in $\sigma(\tau)$ such that 
    \begin{equation*}
        M_{\infty}(R_{\mu}/\Fil_{\otimes}^2R_{\mu})\cong M_{\infty}\bigbra{\ovl{\sigma}^0(\tau)/\rad^2\ovl{\sigma}^0(\tau)},
    \end{equation*}
    where $\ovl{\sigma}^0(\tau)$ is the reduction modulo $p$ of $\sigma^0(\tau)$ and $\rad^2\ovl{\sigma}^0(\tau)\eqdef\rad\bigbra{\rad\ovl{\sigma}^0(\tau)}$ is the radical of the radical of $\ovl{\sigma}^0(\tau)$ as an $\FF[\GL_2(\cO_{F_v})]$-module. By the notation of the proof of \cite[Prop.~8.2.3]{BHHMS} based on \cite[\S10.1]{EGS15}, the representation $\ovl{\sigma}^0(\tau)/\rad^2\ovl{\sigma}^0(\tau)$ has the form $\ovl{\sigma}^{\cJ_0}$ for some capped interval $\cJ_0\subseteq\sset{0,\ldots,f-1}$. Hence $M_{\infty}\bigbra{\ovl{\sigma}^0(\tau)/\rad^2\ovl{\sigma}^0(\tau)}$ is free of rank $r$ over its scheme-theoretic support by the proof of \cite[Prop.~8.2.3]{BHHMS}.
    
    Next we show that if $I\subseteq S$ is such that $\big|I\cap\set{\pm\omega^{(i)}}\big|+\big|S_{\rhobar}^{\sigma}\cap\set{\pm\omega^{(i)}}\big|=1$ for each $0\leq i\leq f-1$, then $M_{\infty}(\wt{R}_{\mu,I})$ is free of rank $r$ over its scheme-theoretic support, which is   $R_{\infty}\otimes_{R_{\rhobar}}R_{\rhobar}^{T_{\sigma,I}}$ (see \cite[Prop.~4.6,~Prop.~4.7]{Le19}). Here $S=\set{\pm\omega^{(i)}}_{0\leq i\leq f-1}$ whose subsets parametrize the Serre weights that appear in $R_{\mu}$, $\rhobar$ is the same as $\rbar_v^{\vee}$, $S_{\rhobar}^{\sigma}$ is a subset of $S$ satisfying $W(\rhobar)=\set{\sigma_J|J\subset S_{\rhobar}^{\sigma}}$, $\wt{R}_{\mu}$ is the same as $\wt{P}_{\sigma_v}$, $\wt{R}_{\mu,I}$ is a certain quotient of $\wt{R}_{\mu}$, $T_{\sigma,I}$ is the set of tame inertial types that appear as subquotients in the $\Gamma$-representation $\wt{R}_{\mu,I}[1/p]$ over $E$ and $R_{\rhobar}^{T_{\sigma,I}}$ parametrizes potentially crystalline framed deformations of $\rhobar$ of Hodge--Tate weights $(1,0)$ at each embedding and inertial type in $T_{\sigma,I}$. In fact, the argument of \cite[Prop.~4.6]{Le19} shows that $M_{\infty}(\wt{R}_{\mu,I})$ is minimally generated by $r$ elements, and the argument of \cite[Prop.~4.7]{Le19} shows that $M_{\infty}(\wt{R}_{\mu,I})$ has scheme-theoretic support $R_{\infty}\otimes_{R_{\rhobar}}R_{\rhobar}^{T_{\sigma,I}}$ which is $p$-torsion free. Moreover, by exactness of $M_{\infty}$ we have
    \begin{equation*}
        M_{\infty}\big(\wt{R}_{\mu,I}\big)[1/p]\cong\bigoplus\limits_{\tau\in T_{\sigma,I}}M_{\infty}(\sigma^0(\tau))[1/p]
    \end{equation*}
    for any choices of $\GL_2(\cO_{F_v})$-stable $W(\FF)$-lattices $\sigma^0(\tau)\subset\sigma(\tau)$. In particular, we can take $\sigma^0(\tau)$ to have irreducible cosocle (see \cite[Lemma~4.1.1]{EGS15}). By (ii) and the fact that the supports of $M_{\infty}(\sigma^0(\tau))[1/p]$ are pairwise disjoint for $\tau\in T_{\sigma,I}$, it follows that  $M_{\infty}\big(\wt{R}_{\mu,I}\big)[1/p]$ is free of rank $r$ over $\big(R_{\infty}\otimes_{R_{\rhobar}}R_{\rhobar}^{T_{\sigma,I}}\big)[1/p]$. We conclude that there is a surjection $\bbra{R_{\infty}\otimes_{R_{\rhobar}}R_{\rhobar}^{T_{\sigma,I}}}^r\onto M_{\infty}(\wt{R}_{\mu,I})$, which is an isomorphism when inverting $p$ by \cite[Thm.~2.4]{Mat89}, hence is also injective since $R_{\infty}\otimes_{R_{\rhobar}}R_{\rhobar}^{T_{\sigma,I}}$ is $p$-torsion free.

    %We deduce from Lemma \ref{Global Lem free} that $M_{\infty}(\wt{R}_{\mu,I})$ is free is rank $r$ over its scheme-theoretic support, which is  $R_{\infty}\otimes_{R_{\rhobar}}R_{\rhobar}^{T_{\sigma,I}}$.
    
    Then we show that if $I\subseteq S$ is such that $\big|I\cap\set{\pm\omega^{(i)}}\big|+\big|S_{\rhobar}^{\sigma}\cap\set{\pm\omega^{(i)}}\big|\leq1$ for each $0\leq i\leq f-1$, then $M_{\infty}(\wt{R}_{\mu,I})$ is free of rank $r$ over its scheme-theoretic support, which is   $R_{\infty}\otimes_{R_{\rhobar}}R_{\rhobar}^{T_{\sigma,I}}$ (see \cite[Thm.~4.9]{Le19}). The proof is completely analogous to that of \cite[Thm.~4.9]{Le19}.
    
    In particular if we take $I=\emptyset$ so that we have $\wt{R}_{\mu,\emptyset}=\wt{R}_{\mu}=\wt{P}_{\sigma_v}$, we get that $M_{\infty}(\wt{P}_{\sigma_v})$ is free of rank $r$ over its scheme-theoretic support, which is $R_{\infty}\otimes_{R_{\rhobar}}R_{\rhobar}^{T_{\sigma,\emptyset}}$. Moreover, by the Chinese reminder theorem we have
    \begin{equation*}
        \big(R_{\infty}\otimes_{R_{\rhobar}}R_{\rhobar}^{T_{\sigma,\emptyset}}\big)[1/p]\cong\bigoplus_{\tau}\big(R_{\infty}/\fp_{\tau}\big)[1/p]\cong\big(R_{\infty}/\cap_{\tau}\fp_{\tau}\big)[1/p],
    \end{equation*}
    where $\tau$ runs over all tame inertial types such that $\sigma_v\in\JH(\ovl{\sigma(\tau)})$. Since both $R_{\infty}/\fp_{\tau}$ and $R_{\infty}^{(1,0),\tau}$ are $p$-torsion free $\cO$-algebras, we deduce that $R_{\infty}\otimes_{R_{\rhobar}}R_{\rhobar}^{T_{\sigma,\emptyset}}\cong R_{\infty}/\cap_{\tau}\fp_{\tau}.$  Finally, each $\fp_{\tau}$ is a prime ideal because $R_{\infty}/\fp_{\tau}=R_{\infty}^{(1,0),\tau}$ is a domain by \cite[Thm.~7.2.1]{EGS15}.
\end{proof}

Now we fix $\sigma_v\in W(\rbar_v^{\vee})$. We let $R_{2.j}'$ (resp.\,$R$) be the representation of $\GL_2(\cO_{F_v})$ on $W(\FF)$-modules defined as in \cite[p.96]{BHHMS} (resp.\,\cite[(56)]{BHHMS}). In particular, we have $R/pR\cong\Proj_{\wt{\Gamma}}\sigma_v$ (see \cite[Cor.~7.3.4]{BHHMS}). Let $r$ be the integer as in Proposition \ref{Global Prop type}. The following proposition is a generalization of \cite[Thm.~8.3.4]{BHHMS}, \cite[Thm.~8.3.9]{BHHMS},
\cite[Cor.~8.3.10]{BHHMS} and \cite[Thm.~8.3.11]{BHHMS} (where $\rbar_v$ was assumed to be semisimple).

\begin{proposition}\label{Global Prop R}
\begin{enumerate}
    \item 
    For each $0\leq j\leq f-1$, the module $M_{\infty}(R_{2,j}')$ is free of rank $r$ over $R_{\infty}/\cap_{\tau}\fp_{\tau}$, where $\tau$ runs over all tame inertial types such that $\sigma_v\in\JH(\ovl{\sigma(\tau)})$ and $\fp_{\tau}$ is the prime ideal $\Ker(R_{\infty}\onto R_{\infty}^{(2,-1)_j,\tau})$ of $R_{\infty}$, where $(2,-1)_j$ is $(2,-1)$ in the $j$-th embedding and $(1,0)$ elsewhere.
    \item
    The module $M_\infty(R)$ is free of rank $r$ over $R_\infty/\cap_{\lambda,\tau}{\fp}_{\lambda,\tau}$, where $\tau$ runs over all tame inertial types such that $\sigma_v\in \JH(\ovl{\sigma(\tau)})$, $\lambda=(\lambda_{j})_{0\leq j\leq f-1}$ runs over the Hodge--Tate weights such that $\lambda_{j}\in \set{(1,0),(2,-1)}$ for all $j$ and ${\fp}_{\lambda,\tau}$ is the prime ideal $\ker(R_\infty\onto  R_\infty^{\lambda,{\tau}})$ of $R_{\infty}$. In particular, we have $\dim_{\FF}M_\infty(R)/{\fm}_\infty =r$.
    \item
    The surjection
    \begin{equation*}
        \Proj_{\wt{\Gamma}}\sigma_v\onto\sigma_v
    \end{equation*}
    induces an isomorphism of nonzero finite-dimensional $\FF$-vector spaces
    \begin{equation*}
        M_\infty\big(\Proj_{\wt{\Gamma}}\sigma_v\big)/\fm_{\infty}\ism M_\infty(\sigma_v)/\fm_{\infty}.
    \end{equation*}
\end{enumerate}
\end{proposition}

\begin{proof}
    The proof of (i) is analogous to the one of \cite[Thm.~8.3.4]{BHHMS}. All the arguments concerning the set of Serre weights $W(\rbar_v^{\vee})$ go through in the general case because we always have $W(\rbar_v^{\vee})\subseteq W((\rbar_v^{\vee})^{\ss})$ (see (\ref{Comp Eq SW})). We also replace \cite[Prop.~8.2.6]{BHHMS} by Proposition \ref{Global Prop Proj}(iii). Finally each $\fp_{\tau}$ is a prime ideal by Theorem \ref{Comp Thm single}.
    
    The proof of (ii) is analogous to the one of \cite[Thm.~8.3.9]{BHHMS} using the same comment on $W(\rbar_v^{\vee})$ together with Proposition \ref{Global Prop Proj}(iii), Proposition \ref{Comp Prop p} and (i).
    
    Finally, (iii) is a direct consequence of (ii) and Proposition \ref{Global Prop type}(i).
\end{proof}

\hspace{\fill}

We let $\pi'$ be the admissible smooth representation of $\GL_2(F_v)$ defined as in \cite[(65),(66)]{BHHMS}. Recall that we defined the Gelfand--Kirillov dimension $\dim_{\GL_2(F_v)}(\pi)$ in \S\ref{Sec GK}. The following theorem is a generalization of \cite[Thm.~8.4.1]{BHHMS}, \cite[Thm.~8.4.2]{BHHMS} and \cite[Cor.~8.4.6]{BHHMS}, where $\rbar_v$ was assumed to be semisimple.

\begin{theorem}\label{Global Thm main}
\begin{enumerate}
    \item 
    We have $\dim_{\GL_2(F_v)}(\pi')=[F_v:\Qp]$.
    \item
    There is an integer $r\geq 1$ such that
    \begin{equation*}
        \pi'[\fm_{K_1}^2]\cong \big(\wt{D}_0(\rbar_v^\vee)\big)^{\oplus r},
    \end{equation*}
    where $\wt{D}_0(\rbar_v^\vee)$ is defined in \S\ref{Sec GK}. In particular, each irreducible constituent of $\pi[\fm_{K_1}^2]$ has multiplicity $r$.
    \item
    Let $\pi$ be the admissible smooth representation of $\GL_2(F_v)$ over $\FF$ defined as in \cite[Cor.~8.4.6]{BHHMS}. Then we have $\dim_{\GL_2(F_v)}(\pi)=[F_v:\Qp]$.
\end{enumerate}
\end{theorem}

\begin{proof}
    The proof of (i) is analogous to the one of \cite[Thm.~8.4.1]{BHHMS}. For the upper bound of the Gelfand-Kirillov dimension, we use Theorem \ref{GK Thm upper}. The condition (ii) in Theorem \ref{GK Thm upper} is guaranteed by Proposition \ref{Global Prop R} (iii), and the conditions (i) and (iii) in Theorem \ref{GK Thm upper} are satisfied as in the proof of \cite[Thm.~8.4.1]{BHHMS}.
    
    The proof of (ii) is completely analogous to the one of \cite[Thm~8.4.2]{BHHMS} using Proposition \ref{Global Prop R} (iii). Here we recall that the $\wt{\Gamma}$-representation $\wt{D}_0(\rbar_v^\vee)$ is multiplicity-free.
    
    Finally, the proof of (iii) is completely analogous to that of \cite[Cor~8.4.6]{BHHMS} using (i).
\end{proof}

\subsection*{Acknowledgement}

\hspace{1.5em}%
The author is extremely grateful to Christophe Breuil for suggesting this problem and to Christophe Breuil and Ariane M\'ezard for helpful discussions and for comments on earlier drafts of this paper. The author thanks Stefano Morra for answering questions on computing Galois deformation rings, and thanks Yongquan Hu for answering questions on representation theory.

This work was supported by the Ecole Doctorale de Math\'ematiques Hadamard (EDMH).

%\newpage

\bibliography{bib/1}

\begin{thebibliography}{10}

\bibitem{Bre03}
Christophe Breuil.
\newblock Sur quelques repr\'{e}sentations modulaires et {$p$}-adiques de
  {${\rm GL}_2(\bold Q_p)$}. {I}.
\newblock {\em Compositio Math.}, 138(2):165--188, 2003.

\bibitem{BHHMS}
Christophe {Breuil}, Florian {Herzig}, Yongquan {Hu}, Stefano {Morra}, and
  Benjamin {Schraen}.
\newblock Gelfand-kirillov dimension and mod p cohomology for $\mathrm{GL}_2$.
\newblock {\em arXiv preprint arXiv:2009.03127}.

\bibitem{BM02}
Christophe Breuil and Ariane M\'{e}zard.
\newblock Multiplicit\'{e}s modulaires et repr\'{e}sentations de {${\rm
  GL}_2({\bf Z}_p)$} et de {${\rm Gal}(\overline{\bf Q}_p/{\bf Q}_p)$} en
  {$l=p$}.
\newblock {\em Duke Math. J.}, 115(2):205--310, 2002.
\newblock With an appendix by Guy Henniart.

\bibitem{BP12}
Christophe Breuil and Vytautas Pa\v{s}k\={u}nas.
\newblock Towards a modulo {$p$} {L}anglands correspondence for {${\rm GL}_2$}.
\newblock {\em Mem. Amer. Math. Soc.}, 216(1016):vi+114, 2012.

\bibitem{BDJ10}
Kevin Buzzard, Fred Diamond, and Frazer Jarvis.
\newblock On {S}erre's conjecture for mod {$\ell$} {G}alois representations
  over totally real fields.
\newblock {\em Duke Math. J.}, 155(1):105--161, 2010.

\bibitem{CL18}
Ana Caraiani and Brandon Levin.
\newblock Kisin modules with descent data and parahoric local models.
\newblock {\em Ann. Sci. \'{E}c. Norm. Sup\'{e}r. (4)}, 51(1):181--213, 2018.

\bibitem{Col10}
Pierre Colmez.
\newblock Repr\'{e}sentations de {${\rm GL}_2(\bold Q_p)$} et
  {$(\phi,\Gamma)$}-modules.
\newblock {\em Ast\'{e}risque}, (330):281--509, 2010.

\bibitem{Dee01}
Jonathan Dee.
\newblock {$\Phi$}-{$\Gamma$} modules for families of {G}alois representations.
\newblock {\em J. Algebra}, 235(2):636--664, 2001.

\bibitem{Eme11}
Matthew Emerton.
\newblock Local-global compatibility in the p-adic {L}anglands programme for
  $\mathrm{GL}_{2/\mathbb{Q}}$.
\newblock {\em preprint}, 2011.

\bibitem{EGS15}
Matthew Emerton, Toby Gee, and David Savitt.
\newblock Lattices in the cohomology of {S}himura curves.
\newblock {\em Invent. Math.}, 200(1):1--96, 2015.

\bibitem{Fon90}
Jean-Marc Fontaine.
\newblock Repr\'{e}sentations {$p$}-adiques des corps locaux. {I}.
\newblock In {\em The {G}rothendieck {F}estschrift, {V}ol. {II}}, volume~87 of
  {\em Progr. Math.}, pages 249--309. Birkh\"{a}user Boston, Boston, MA, 1990.

\bibitem{HW}
Yongquan Hu and Haoran Wang.
\newblock On the mod $p$ cohomology for $\mathrm{GL}_2$: the non-semisimple
  case.
\newblock {\em Cambridge Journal of Mathematics}, 10(2):261--431, 2022.

\bibitem{Kis08}
Mark Kisin.
\newblock Potentially semi-stable deformation rings.
\newblock {\em J. Amer. Math. Soc.}, 21(2):513--546, 2008.

\bibitem{Le19}
Daniel {Le}.
\newblock Multiplicity one for wildly ramified representations.
\newblock {\em Algebra Number Theory}, 13(8):1807--1827, 2019.

\bibitem{LLHL19}
Daniel Le, Bao~V. Le~Hung, and Brandon Levin.
\newblock Weight elimination in {S}erre-type conjectures.
\newblock {\em Duke Math. J.}, 168(13):2433--2506, 2019.

\bibitem{LLHLM}
Daniel Le, Bao~V. Le~Hung, Brandon Levin, and Stefano Morra.
\newblock Local models for galois deformation rings and applications.
\newblock {\em arXiv preprint arXiv:2007.05398}.

\bibitem{LLHLM18}
Daniel Le, Bao~V. Le~Hung, Brandon Levin, and Stefano Morra.
\newblock Potentially crystalline deformation rings and {S}erre weight
  conjectures: shapes and shadows.
\newblock {\em Invent. Math.}, 212(1):1--107, 2018.

\bibitem{LLHLM20}
Daniel Le, Bao~V. Le~Hung, Brandon Levin, and Stefano Morra.
\newblock Serre weights and {B}reuil's lattice conjecture in dimension three.
\newblock {\em Forum Math. Pi}, 8:e5, 135, 2020.

\bibitem{LMS20}
Daniel {Le}, Stefano {Morra}, and Benjamin {Schraen}.
\newblock Multiplicity one at full congruence level.
\newblock {\em J. Institut Math. Jussieu}, pages 1--22, 2020.

\bibitem{Mat89}
Hideyuki Matsumura.
\newblock {\em Commutative ring theory}, volume~8 of {\em Cambridge Studies in
  Advanced Mathematics}.
\newblock Cambridge University Press, Cambridge, second edition, 1989.
\newblock Translated from the Japanese by M. Reid.

\end{thebibliography}
\bibliographystyle{plain}

\end{document}